\documentclass[12pt,reqno,a4paper]{amsart}
\usepackage[svgnames]{xcolor}

\usepackage[latin1]{inputenc}
\usepackage{tikz}

\usetikzlibrary{arrows.meta}

\usepackage{nicematrix}
\usepackage{pst-node}
\usepackage{censor}
\usetikzlibrary{fit}

\usepackage[linkcolor=red,colorlinks=true]{hyperref}
\usepackage{soul}

\usepackage{comment}
\usepackage{mdframed}
\usepackage{setspace}
\usepackage{mathrsfs}  
\usepackage{MnSymbol}
\usepackage{etoolbox}
\usepackage{amsmath}
\usepackage{enumerate}
\usepackage{amscd}
\usepackage{amsthm}
\usepackage{graphicx}
\usepackage[all,cmtip]{xy}
\usepackage[normalem]{ulem}

\patchcmd{\subsection}{-.5em}{.5em}{}{}
\patchcmd{\subsubsection}{-.5em}{.5em}{}{}

\makeatother
\usepackage{hyperref}

\numberwithin{equation}{section}

\let\epsilon\varepsilon
\renewcommand\emptyset\varnothing
\let\phi\varphi

\newcommand{\SL}{\operatorname{SL}}

\newcommand{\Aut}{\operatorname{Aut}}

\newcommand{\I}{\mathrm{I}}
\newcommand{\II}{\mathrm{II}}
\newcommand{\III}{\mathrm{III}}

\newcommand{\mfg}{[\cR]}
\newcommand{\ns}{\Aut(X,[\mu])}
\newcommand{\malg}{\mathrm{MAlg}}

\newcommand{\cB}{\mathcal{B}}

\newcommand{\cO}{\mathcal{O}}
\newcommand{\cP}{\mathcal{P}}

\newcommand{\cR}{\mathcal{R}}

\newcommand{\cU}{\mathcal{U}}
\newcommand{\cV}{\mathcal{V}}

\newcommand{\cZ}{\mathcal{Z}}

\newcommand{\bC}{\mathbb{C}}

\newcommand{\bN}{\mathbb{N}}

\newcommand{\bQ}{\mathbb{Q}}
\newcommand{\bR}{\mathbb{R}}

\newcommand{\bZ}{\mathbb{Z}}

\newcommand{\qand}{\quad \textrm{and} \quad}

\newcommand\subsetsim{\mathrel{%
\ooalign{\raise0.2ex\hbox{$\subset$}\cr\hidewidth\raise-0.8ex\hbox{\scalebox{0.9}{$\sim$}}\hidewidth\cr}}}
\newcommand{\eps}{\varepsilon}

\DeclareMathOperator{\Hom}{Hom}

\DeclareMathOperator{\supp}{supp}

\DeclareMathOperator{\Sym}{Sym}
\DeclareMathOperator{\Stab}{Stab}

\DeclareMathOperator{\Env}{Env}
\DeclareMathOperator{\Miss}{Miss}
\DeclareMathOperator{\Nbh}{U}
\DeclareMathOperator{\Fix}{Fix}

\DeclareMathOperator{\dom}{dom}
\DeclareMathOperator{\ran}{ran}

\DeclareMathOperator{\Z}{\bZ}
\DeclareMathOperator{\R}{\bR}

\newcommand{\Boom}{\mathrm{Boom}}
\newcommand{\Sub}{\mathrm{Sub}}

\newcommand\abs[1]{\left|#1\right|}
\renewcommand{\epsilon}{\varepsilon}

\definecolor{lichtgrijs}{gray}{0.95}
\mdfdefinestyle{mystyle}{ %
    backgroundcolor=lichtgrijs, %
    linewidth=0pt, %
    innertopmargin=10pt, %
    innerbottommargin=10pt,%
    nobreak=true
}

\theoremstyle{theorem}
\newtheorem{theorem}{Theorem}

\newtheorem*{theorem*}{Theorem}

\newtheorem*{corollary*}{Corollary}

\newtheorem*{proposition*}{Proposition}

\newtheorem*{lemma*}{Lemma}

\theoremstyle{definition}

\newtheorem*{definition*}{Definition}

\newtheorem*{remark*}{Remark}
\newtheorem*{construction*}{Construction}

\newtheorem{no}[subsubsection]{\S}

\newtheorem*{example*}{Example}
\newtheorem{question}[theorem]{Question}
\newtheorem*{question*}{Question}

\usepackage{changepage}
\usepackage{mathtools}
\usepackage{calc}

\usepackage[toc,page]{appendix}

\usepackage{scalerel}

\usepackage{extarrows}

\DeclareMathSymbol{\shortminus}{\mathbin}{AMSa}{"39}
\usepackage{mathtools}
\begin{document}

\title[Elementwise conservative actions]{Elementwise conservative actions and new constructions of boomerang subgroups}
\author{Yair Glasner}
\address{Yair Glasner, Department of Mathematics,
Ben-Gurion University of the Negev
Be'er Sheva, 84105, Israel,
{\tt yairgl@bgu.ac.il}
}

\author{Tobias Hartnick}
\address{Tobias Hartnick, Institut f\"ur Algebra und Geometrie, KIT, Englerstra{\ss}e 2, 76131 Karlsruhe, Germany,
{\tt tobias.hartnick@kit.edu}
}
\author{Waltraud Lederle} 
\address{Waltraud Lederle, Faculty of Mathematics, Bielefeld University, Universit\"atsstra{\ss}e 25, 33501 Bielefeld, Germany,
{\tt waltraud.lederle\@@uni-bielefeld.de}
}

\begin{abstract}
    We show that countable non-abelian free groups admit uncountably many mutually singular elementwise conservative non-singular random subgroups, which are supported on infinite subgroups of infinite index and singular with respect to every invariant random subgroup. This complements recent rigidity results for elementwise-conservative random subgroups in higher rank lattices by the first- and third-named authors. Our proof is based on a study of representations of free groups into measurable full groups in which the action of the first generator of the free group is fixed. We show that elementwise conservativity is generic among such representations in the sense of Baire category.
\end{abstract}
\maketitle
\tableofcontents

\section{Introduction}

\subsection{Dynamical rigidity and flexibility}

One of the main themes of modern group theory is the dichotomy between \emph{flexibility} and \emph{rigidity}. For example, lattices in $\mathrm{SL}_2(\R)$ admit a rich space of deformations (closely related to Teichm\"uller space) whereas lattices in $\mathrm{SL}_2(\bC)$ or $\mathrm{SL}_3(\R)$ do not admit any non-trivial deformations at all - this is Mostow's celebrated rigidity theorem. Lattices in higher rank simple Lie groups (such as $\mathrm{SL}_3(\R)$) satisfy a whole plethora of rigidity properties, from Kazhdan's Property (T) and various strengthenings thereof to Margulis' superrigidity and its generalizations.

It was observed at an early stage that rigidity properties of higher rank lattices $\Lambda$ are closely connected to \emph{dynamical rigidity properties} of the action of the ambient Lie group $G$ on the homogeneous space $G/\Lambda$, which preserves an invariant probability measure. This lead to the study of rigidity properties of general \emph{probability measure preserving (p.m.p.) actions} of higher rank Lie groups. Notably, Nevo, Stuck and Zimmer \cite{SZ:94},\cite{NZ:IFT} established that every ergodic p.m.p. action of a center-free simple higher rank Lie group $G$ is either essentially free or transitive. 

The Nevo--Stuck--Zimmer theorem can be stated equivalently in the language of conjugation-invariant probability measures on the Chabauty space $\mathrm{Sub}(G)$ of closed subgroups of $G$, a.k.a.\ \emph{invariant random subgroups} (IRSs). Indeed, since every p.m.p. action $G \curvearrowright (X, \mu)$ gives rise to an IRS by pushing $\mu$ forward via the \emph{stabilizer map} 
\[ \mathrm{Stab}\colon X \to \mathrm{Sub(G)}, \quad x \mapsto \mathrm{Stab}_G(x)\] 
and since every IRS arises from a p.m.p.\ action in this way \cite{AbertGlasnerVirag2014,7sam}, the Nevo--Stuck--Zimmer theorem says that every ergodic IRS of a center-free simple higher rank Lie group is either trivial or supported on the conjugation orbit of a lattice. In particular, a generic point of such an IRS is either trivial or a lattice; similar rigidity results also hold for IRSs of higher rank lattices. 

On the contrary, rank one Lie groups (and, more generally, hyperbolic locally compact groups) and their lattices admit plenty of non-trivial IRSs \cite{7sam_2} - in this sense they are \emph{dynamically flexible}. This flexibility already shows in the easiest non-trivial example, namely non-abelian finitely generated free groups. 

In recent years it has become more and more apparent that dynamical rigidity results in higher rank concern not only p.m.p. actions but also other classes of non-singular actions. Notably, motivated by applications to Margulis' unbounded injectivity radius conjecture, Gelander and Fraczyk \cite{GelanderFraczyk} initiated a study of \emph{stationary random subgroups} (SRSs), i.e. stationary probability measures on Chabauty spaces, and established a version of the Nevo--Stuck--Zimmer theorem in this context. 

Similarly, the work of the first and the third author in \cite{GlasnerLederle2025} implies a version of the Nevo--Stuck--Zimmer theorem for a different class of non-singular random subgroups which we call \emph{elementwise conservative random subgroups} (ECRSs). The present article, which can be seen as a companion article to this work, is devoted to an instance of \emph{dynamical flexibility of ECRSs}. Namely, we are going to show that a countable non-abelian free group admits plenty of ECRSs which are singular with respect to every IRS.

\subsection{ECRSs and boomerang subgroups}

In order to state our first main result, we now introduce the class of non-singular systems we will be working with; see Subsection \ref{SubsecNonSing} below for general background and terminology concerning non-singular transformations. Given a discrete group $\Lambda$, we denote by $\mathrm{Sub}(\Lambda)$ the space of subgroups of $\Lambda$, which is a compact $\Lambda$-space when equipped with the action by conjugation and the topology as a closed subset of $\{0,1\}^{\Lambda}$. Recall that a non-singular transformation $T$ is called \emph{conservative} if every $T$-wandering set is a nullset (see Subsection \ref{SubsecConservative} for various equivalent characterizations).
\begin{definition*} 
Let $\Lambda$ be a countable discrete group.
\begin{enumerate}[(i)]
 \item A non-singular action of $\Lambda$ on a standard probability space $(X, \mu)$ is called \emph{elementwise conservative} (or an \emph{e.c.\ action} for short) if every $\gamma \in \Lambda$ acts by a conservative transformation.
 \item A probability measure $\mu$ on $\mathrm{Sub}(\Lambda)$ is called an \emph{e.c.\ random subgroup (ECRS)} if the conjugation action of $\Lambda$ on $(\mathrm{Sub}(\Lambda), \mu)$ is non-singular and e.c..
 \item A subgroup $\Delta \in \mathrm{Sub}(\Lambda)$ is called a \emph{boomerang subgroup} if for every $\gamma \in \Lambda$ there exists a sequence $n_k \to \infty$ with $\gamma^{n_k} \Delta \gamma^{-n_k} \to \Delta$ in $\mathrm{Sub}(\Lambda)$.
 \end{enumerate}
\end{definition*} 
 It follows from the Poincar\'e recurrence theorem that every p.m.p. action is e.c., and hence every IRS is an ECRS. Moreover, almost every instance of an ECRS (and hence of an IRS) is a boomerang subgroup. Schematically,
\[\begin{xy}\xymatrix{
\text{p.m.p. action} \ar@{=>}[d] \ar[rrr]^{\mathrm{stabilizer}} &&&\text{IRS} \ar@{=>}[d] \\
\text{e.c.\ action} \ar[rrr]^{\text{stabilizer}} &&& \text{ECRS} \ar[rrr]^{\text{almost surely}\hphantom{ppppppppp}} &&& \text{boomerang subgroup.}
}
\end{xy}
\]
Note that while IRSs and ECRSs are concepts from non-singular dynamics, boomerang subgroups belong to the realm of topological dynamics.
\begin{remark*} \label{rem:golb_conservativity} Unlike ergodicity or stationarity, the e.c.\ property is an \emph{elementwise} property of an action. Asking for the action to be \emph{globally} conservative (in the sense that for every positive measure set $A$ there is some $1 \ne \gamma \in \Lambda$ such that $\mu(A \cap \gamma(A)) > 0$) is a much weaker notion. In fact, as shown in Lemma \ref{lem:whole_gp}, \emph{every} non-singular random subgroup is globally conservative, whereas elementwise conservativity is much more restrictive.

On first sight, an elementary conservative action is just a countable collection of conservative $\Z$-actions, one for every element of the group. However, the fact that these different actions combine into an action of some given group $\Lambda$ often puts severe constraints on possible e.c.\ actions of $\Lambda$ as illustrated by the following rigidity result from \cite{GlasnerLederle2025}.
\end{remark*} 
\begin{theorem*}[{\cite[Corollary 1.3]{GlasnerLederle2025}}] 
Let $\Lambda$ be a lattice in a center-free simple Lie group $G$, and assume that $\operatorname{rk}_{\bQ}(\Lambda) \ge 2$. Then, every nontrivial ergodic ECRS of $\Lambda$ is supported on finite index subgroups. In fact, every nontrivial boomerang subgroup of $\Lambda$ is of finite index.
\end{theorem*}
 This rigidity result is in line with the general philosophy explained above: One expects ECRSs (and hence boomerang subgroups) to be rare in higher rank situations, but abundant in rank one. The goal of the present article is to partially confirm the expectation of dynamical flexibility of ECRSs in rank one by showing:
\begin{theorem}[Existence of exotic ECRSs]\label{ExECRS} Let $\Gamma$ be a countable non-abelian free group. Then $\Gamma$ admits uncountably many mutually singular ergodic ECRSs which are singular with respect to every IRS and supported on core-free, co-amenable  and co-highly transitive subgroups.
\end{theorem}
Here, a subgroup $\Delta < \Gamma$ is called \emph{core-free} if $\bigcap_{\gamma \in \Gamma} \gamma \Delta \gamma^{-1} = \{1\}$; it is called \emph{co-amenable} if the left action $\Gamma \curvearrowright \Gamma/\Delta$ is amenable in the sense of Greenleaf, i.e.\ admits $(F,\epsilon)$-almost invariant sets for every finite subset $F \subset \Gamma$ and every $\epsilon >0$. Finally it is called \emph{co-highly transitive} if the action of $\Gamma \curvearrowright \Gamma/\Delta$ is transitive on ordered $k$-tuples of distinct points, for every $k \in \bN$. 

Every core-free subgroup of an infinite group is of infinite index, every co-amenable subgroup of a non-amenable group is infinite, and every co-highly transitive group is maximal. Thus Theorem \ref{ExECRS} produces exotic boomerangs which are infinite, of infinite index and maximal in the ambient free group.

\subsection{Representations into measurable full groups}

We now outline a general strategy towards the proof of Theorem \ref{ExECRS}. Thus let $\Gamma$ be a countable non-abelian free group.  Firstly, in order to construct an ECRS as in the statement of the theorem, it suffices to construct a non-singular action of $\Gamma$ on a standard probability space $(X, \mu)$ with the following additional properties:
\begin{enumerate}[(1)]
\item The action is faithful and e.c., but there is no $\Gamma$-invariant probability measure in the measure class of $\mu$.
\item The stabilizer map $\mathrm{Stab} \colon X \to \mathrm{Sub}(\Gamma)$ is injective (hence a Borel isomorphism onto its image).
\item Almost all stabilizers are core-free, co-amenable and co-highly transitive.
\end{enumerate}
Indeed, in this case the push-forward of $\mu$ under the stablizer map will be an ECRS with the desired properties. Most of our work goes into constructing a non-singular action satisfying (1); showing that these actions also satisfy (2) and (3) is comparatively easy. The construction is based on three main ingredients.

The first ingredient is the technology of measurable full groups (see Subsection \ref{SubsecMFG} and the references therein). We start from a hyperfinite and ergodic countable Borel equivalence relation $\cR$ on 
a standard probability space $(X, \mu)$ of Krieger type $\II_\infty$ or $\III_\lambda$ for some $\lambda \in [0,1]$ (see \ref{KriegerOE}) and consider the associated measurable full group $G$  (see \ref{MFG}). Here the assumptions on the Krieger type are chosen such that $\mu$ is non-atomic and such that there is no $G$-invariant probability measure in the class of $\mu$. Our goal is now to find representations $\rho\colon \Gamma \to G$ with dense image such that $\rho(\gamma)$ is conservative for every $\gamma \in \Gamma$; the corresponding $\Gamma$-actions on $(X, \mu)$ will then satisfy (1). The question now becomes how to construct such representations.

Here is where the second ingredient comes in: Baire category theory. Instead of considering individual representations $\rho\colon \Gamma \to G$ we consider the Polish space $\mathrm{Hom}(\Gamma, G)$ of \emph{all} such representations. We say that a \emph{generic representation} has a Property (P) if the set of representations with this property contains a dense $G_\delta$-subset of $\mathrm{Hom}(\Gamma, G)$. By the Baire category theorem, if $(P_n)_{n \in \bN}$ is a countable set of properties that hold for a generic representation, then a generic representation satisfies all of these properties simultaneously. In particular, there exists a representation satisfying $(P_n)$ for all $n \in \bN$, and this allows for the construction of representations with peculiar properties. Unfortunately it turns out that, for our finitely-generated non-abelian free group, generic representations into measurable full groups have finite orbits; while this does imply that generic representations are e.c., the resulting boomerang subgroups are of finite index and thus rather boring. 

To fix this problem, we use a third ingredient, which is inspired from \cite{Bowen:free_irs,EG:generic_irs}: constraining representations.
We fix a generator $T$ for $\cR$ and a (finite or countably infinite) basis $(a, b_1, b_2, \dots)$ for $\Gamma$. We then say
that a representation $\rho\colon \Gamma \to G$ is \emph{$T$-constrained} if $\rho(a) = T$ (in particular almost all orbits of $a$ are infinite). The closed subspace $\mathrm{Hom}_T(\Gamma, G) \subset \mathrm{Hom}(\Gamma, G)$ of $T$-constrained representations is Polish, hence we can talk about \emph{generic $T$-constrained representations}. Theorem \ref{ExECRS} is now a consequence of the following main result:

\begin{theorem}[Generic constrained representations]\label{GCR1}
Let $\cR$ be a hyperfinite and ergodic countable Borel equivalence relation on a non-atomic standard probability space $(X, \mu)$ with measurable full group $G$. Then for every choice of generator $T$ of $\cR$, a generic $T$-constrained representation $\rho \colon \Gamma \to G$ has the following properties:
\begin{enumerate}
\item \label{itm:nonatomic inj} $\rho$ is faithful. 
\item \label{itm:nonatomic ec} The action of $\Gamma$ via $\rho$ is elementwise conservative.
\item \label{itm:nonatomic afht} For almost every $x \in X$, the action of $\Gamma$ on the equivalence class $[x]$ is amenable (in the sense of Greenleaf), faithful and highly transitive.
\item \label{itm:nonatomic tnf} The stabilizer map $\Stab \colon (X,\mu) \rightarrow (\Sub(\Gamma),\Stab_*(\mu))$ is an isomorphism of $\Gamma$-systems. 
\end{enumerate}
Moreover, if $\cR$ is not of Krieger type $\II_1$, then:
\begin{enumerate}\setcounter{enumi}{4}
\item \label{NoInvMeasure} There is no $\Gamma$-invariant probability measure in  $[\mathrm{Stab}_*\mu]$.
\item \label{itm:nonatomic dense}  $\rho(\Gamma)$ is dense in $G$.
\end{enumerate}
In fact, the constrained representations with these properties form a dense $G_\delta$-set.
\end{theorem}
Our main contribution is the proof of elementwise conservativity, which is essentially combinatorial and works by reduction to the study of special permutations of $\bZ$. The remainder of the proof is 
heavily inspired from work of Le Ma\^itre and his co-authors (see \cite{LeMaitre2024} and the references therein).

\subsection{Variants and refinements}
In Theorem \ref{GCR1} we have assumed that $\mu$ is non-atomic and ergodic; equivalently, $\cR$ is of Krieger type $\II_1$, $\II_\infty$ or $\III_\lambda$ for some $\lambda \in [0,1]$. Let us briefly explain what happens in the remaining Krieger types:

If $\cR$ is of Krieger type $\I_n$ with $1 \leq n < \infty$, then $[\cR] = \mathrm{Sym}_n$ is a finite symmetric group, and thus all representations $\Gamma \to [\cR]$ are finite.

The case of Krieger type $\I_\infty$ yields something more interesting: Any system $(X, [\mu], T)$ of type $\I_\infty$ is conjugate to $(\bZ, \mu_c, \sigma)$, with $\mu_c$ equivalent to a counting measure on $\bZ$ and $\sigma \colon \Z \rightarrow \Z$ given by
$n \mapsto n+1$. We may (and will) thus assume that $(X, [\mu], T) = (\bZ, \mu_c, \sigma)$, which implies that the associated measurable full group is given by $G = \Sym(\bZ)$, the full symmetric group of $\Z$. Note that the action of $\sigma$ is \emph{not} conservative, since every singleton is a wandering set, and hence $\sigma$-constrained representations into $\mathrm{Sym}(\bZ)$ can never be elementwise conservative. However, the following variant of Theorem \ref{GCR1} still guarantees that stabilizers of generic constrained representations are boomerang subgroups; this has to do with the fact that the non-conservative elements do not stabilize any points.
\begin{theorem}\label{thm:S_infty}
    A generic constrained representation $\rho \in \Hom_\sigma(\Gamma, \Sym(\bZ))$ and its induced $\Gamma$-action on $\Z$ have the following properties:
    \begin{enumerate}
         \item\label{itm:atomic finorb} For an element $w \in \Gamma$ the following are equivalent:
            \begin{enumerate}[(i)]
                \item $w$ is not conjugate to $a^k$ for any $0 \ne k \in \bZ$. 
                \item All orbits of $\rho(w)$ are finite.
                \item $\rho(w)$ is conservative.
            \end{enumerate}
        \item \label{itm:atomic tnf} The stabilizer map $\Stab \colon (X,\mu) \rightarrow (\Sub(\Gamma),\Stab_*(\mu))$ is an isomorphism of $\Gamma$-systems. 
         \item \label{itm:atomic boomerang} For every $x \in \Z$ the stabilizer $\Stab_{\Gamma}(x)$ is a boomerang subgroup of $\Gamma$. 
        \item \label{itm: atomic afht} $\Gamma$ acts amenably, faithfully and highly transitively on $\bZ$.
    \end{enumerate}
\end{theorem}
Ien the case where $\Gamma$ is of countably-infinite rank, Theorem \ref{ExECRS} can be established by a much simpler argument. Namely one can work directly with generic representations (rather than restricted generic representations) and establish the following variant of Theorem \ref{GCR1}:
\begin{theorem} \label{thm:unrestricted}
Let $\Gamma = \langle b_1,b_2,\ldots \rangle$ be a countable non-abelian free group of infinite rank and let $\cR$ be a hyperfinite, ergodic, aperiodic countable Borel equivalence relation on a standard Borel space with measurable full group $G$. Then a generic representation $\rho \colon \Gamma \to G$ and the induced action on $(X,\mu)$ have the following properties: 
\begin{enumerate}
\item \label{itm:p_ep} For every finitely generated subgroup of $\Gamma$ almost all orbits are finite. In particular every element of $\rho(\Gamma)$ is periodic and hence the $\Gamma$-action is elementwise conservative.
\item  \label{itm:p_dense} $\rho(\Gamma)$ is dense in $G$.
\item \label{itm:p_stab} The stabilizer map $\Stab \colon (X,\mu) \rightarrow (\Sub(\Gamma),\Stab_*(\mu))$ is an isomorphism of $\Gamma$-systems. 
In particular, $\Gamma$ does not preserve any invariant probability measure unless $\cR$ is of Krieger type $\II_1$.
\item \label{itm:p_prop} $\Gamma$ acts amenably, faithfully and highly transitively on almost every orbit.
\end{enumerate}
\end{theorem}

\subsection{Related results and open problems}\label{SecRelated}
\begin{no}
Theorem \ref{ExECRS} provides examples of exotic boomerang subgroups which ``do not arise from an IRS'' in one possible sense. The following natural question remains open though: 
\begin{question*} \label{q:orbit_does_not_support_irs} Does there exist an infinite boomerang subgroup of a countable non-abelian free group $\Gamma$ whose orbit closure does not support an IRS?  
\end{question*}
A partial answer to this question follows from recent work of the third-named author \cite{Lederle2025} and work of Gekhtman--Levit \cite{GL19}: There exist infinite boomerang subgroups of $\Gamma$ whose orbit closure only supports a single IRS concentrated on the trivial subgroup. 

More precisely, it was established in \cite[Thm.~1.1]{GL19} (respectively \cite[Thm.~1.4]{GekhtmanLevit2023}) that every non-trival IRS $\Delta$ (respectively every nontrivial SRS $\Delta$ with respect to a nice enough random walk) has critical exponent  $\delta(\Delta) > \frac{\delta(\Gamma)}{2}$ (and even $\delta(\Delta) = \delta(\Gamma)$ if one further assumes that the relevant IRS or SRS is supported on groups of divergence type). On the other hand, the combinatorial construction from \cite{Lederle2025} yields boomerang subgroups of arbitrarily small critical exponent, which therefore do not admit any non-trivial IRS (or non-trivial nice SRS) on their orbit closure. Unfortunately, this argument is not quite strong enough to fully resolve the question since, by a result of Fraczyk \cite[Corollary 3.2]{Fraczyk2019}, the orbit closure of every group $\Delta \in \Sub(\Gamma)$ with $\delta(\Delta) \leq  \frac{\delta(\Gamma)}{2}$ necessarily contains the trivial group. In relation to the current article, also the following question remains open.
\begin{question*} Do the boomerang subgroups of small critical exponent constructed in \cite{Lederle2025} support a non-trivial ECRS on their orbit closure?
\end{question*}
\end{no}

\begin{no} In a different direction, we hope that the notion of an e.c.\ action, as introduced here and in \cite{GlasnerLederle2025}, will prove to be fruitful in nonsingular dynamics beyond the setting of nonfree actions considered here. For example, in a related upcoming paper, Le Ma\^{i}tre and Stalder use a construction similar to ours to obtain free e.c. actions of a free group $\Gamma$ that do not come from p.m.p. actions using Baire generic representations of a free group into the Polish group $\Aut(\R,\lambda)$ of infinite measure preserving transformations of a standard Lebesgue space.
Typically such actions are essentially free and do not contribute new examples of boomerang subgroups, but they do provide a rich source of e.c.\ actions that are not probability measure preserving.
More generally, we propose the following question:
\begin{question*}
Which groups admit interesting e.c. actions with no invariant probability measure? 
\end{question*}
One such example, suggested by Sasha Bontemps and the third named author, goes as follows: Let $\Lambda$ be a nonamenable torsion group acting on its universal strongly proximal flow (a.k.a. its Furstenberg boundary) $\partial_F \Lambda$. For a torsion group every nonsingular action is e.c., but not in an interesting way. Strong proximality rules out the existence of any invariant Borel probability measure on $\partial_F \Lambda$. 
Typically $\partial_F \Lambda$ is far from being a standard Borel space, but it is well known that metric factors $\partial_F \Lambda \rightarrow X$ separate points and one can always pass to such a factor $X$. Indeed take any nonconstant function $f \in C(X)$ - the minimal closed self adjoint $\Lambda$-invariant algebra spanned by $f$ inside $C(\partial_F \Lambda)$ will give rise to such a factor.

The main result of \cite{GlasnerLederle2025} shows that for lattices of higher $\bQ$-rank in simple Lie groups all e.c.\ actions are either finite or essentially free, but it remains open whether there are any e.c. actions of such groups that do not admit an equivalent invariant probability measure.
\end{no}

\subsection{Organization of the article} In Section \ref{SecPrelim} we fix our notation and collect various preliminaries concerning non-singular ergodic theory and measurable full groups. 
In Section \ref{sect:basics} we prove basic properties on representations into measurable full groups.
In Section \ref{SecIInf} we consider restricted representations into the measurable full group of a system of Krieger type $\I_\infty$ and establish Theorem \ref{thm:S_infty}. It turns out that this case is much simpler than that of other Krieger types. Section \ref{Sece.c.} is the main technical section of the article. We first establish that generic constrained representations into the measurable full group are elementwise conservative by reduction to a combinatorial lemma on so-called zebra permutations. The structural insight into generic constrained representations gained by the proof is then exploited to establish further properties of such representations. In Section \ref{SecDense} we establish that generic constrained representations in Krieger types $\II_\infty$ and $\III_\lambda$ are dense by reduction to the aforementioned theorem of Krengel. 
The theorems from the introduction then follow by combining all these results; for the convenience of the reader we have collected their proofs in Section \ref{SecCollect}.

\subsection*{Acknowledgments}
This work was initiated during a visit of TH to YG, and we thank Ben-Gurion University for supporting this visit. WL was partially supported by the YigPrepProp fellowhip by the KIT and was a F.R.S.-FNRS postdoctoral researcher. YG was partially funded by ISF grant 3187/24 as well as by BSF Grant 2024317. Part of this work was completed on a research visit at ETH Zurich, we thank Marc Burger for his hospitality.
The third named author thanks Josh Frisch for helpful conversations.
The authors thank Fran\c{c}ois Le Ma\^{i}tre for helpful discussions and feedback on an earlier version of this manuscript.

\section{Preliminaries}\label{SecPrelim}

\subsection{Polish groups and constrained representations} 
\begin{no}
A topological space is called \emph{Polish} if it is separable and completely metrizable. A topological group $G$ is called a \emph{Polish group} if the underlying space is Polish.

 A subset $Y$ of a Polish space $X$ is called a \emph{$G_\delta$-set}  if it is a countable intersection of open subsets of $X$; equivalently, $Y$ is Polish with respect to the subspace topology. Pre-images of $G_\delta$-sets under continuous map are $G_\delta$, and the \emph{Baire category theorem} states that the class of dense $G_\delta$ subsets is closed under countable intersections. 

A subset of a Polish space $X$ is called \emph{comeager}, or \emph{residual}, if it contains a dense $G_\delta$ set, and \emph{meager} if its complement is comeager. We say that a property $\cP$ of elements of a Polish space $X$ is \emph{generic} if the set of elements of $X$ with this property is comeager. 
\end{no}

\begin{no}\label{FreeGroup}
Throughout this article, $\Gamma$ denotes a countable non-abelian free group. We fix a basis of $\Gamma$ of the form $(a, b_1 \dots, b_{r-1})$ if $\Gamma$ has finite rank $r$, and of the form $(a, b_1, b_2, \dots)$ if $\Gamma$ has countable infinite rank. We then identify elements of $\Gamma$ with reduced words over the alphabet consisting of the basis elements and their inverses. As indicated by the notation, the first generator $a$ will often play a special role in our considerations. If $\Gamma$ has rank at least $r$, then we denote by $\Gamma_r < \Gamma$ the subgroup generated by the first $r$ basis elements $(a, b_1, \dots, b_{r-1})$.
\end{no}

\begin{no}\label{Ev} If $G$ is a Polish group and $T \in G$, then we denote by $\mathrm{Hom}(\Gamma, G)$ the space of homomorphisms of $\Gamma$ to $G$ and by $\mathrm{Hom}_T(\Gamma, G)$ the subspace of those homomorphisms $\rho$ with $\rho(a) = T$. We refer to elements of $\mathrm{Hom}_T(\Gamma, G)$ as \emph{$T$-constrained representations}.

If $\Gamma$ has finite rank $r$, then we topologize these spaces by demanding that the natural bijections $\mathrm{Hom}(\Gamma, G) \to G^r$ and $\mathrm{Hom}_T(\Gamma, G) \to G^{r-1}$ given respectively by $\rho \mapsto (\rho(a), \rho(b_1), \dots, \rho(b_{r-1}))$ and $\rho \mapsto (\rho(b_1), \dots, \rho(b_{r-1}))$ are homeomorphisms; if $\Gamma$ has infinite rank, then we similarly identify $\mathrm{Hom}(\Gamma, G)$ and  $\mathrm{Hom}_T(\Gamma, G)$ with the countable product $G^\infty$ and equip them with the topology corresponding to the product topology on $G^\infty$. In either case,  $\mathrm{Hom}(F_\infty, G)$ and  $\mathrm{Hom}_T(\Gamma, G)$ are Polish, and for every $w\in \Gamma$ the \emph{evaluation map} \[\mathrm{ev}_w\colon \mathrm{Hom}(\Gamma, G) \to G, \quad \rho \mapsto \rho(w)\]
is continuous.

By assumption, the topology on $G$ is metrizable, and we will always fix a bounded metric $d$ on $G$ (not necessarily complete) which induces the given topology. Then for $1 \leq r \leq \infty$ the topology on $G^r$ is induced by the metric 
$d((g_i),(h_i)) = \sum_{i = 1}^r d(g_i,h_i)/2^i$, and we equip $\Hom(\Gamma,G)$ and $\Hom_T(\Gamma,G)$ with the corresponding metrics.
\end{no}
\begin{no} We are going to be interested in ``generic properties'' of constrained representations. If $\Gamma$ is a countable free group, $G$ is a Polish group, $T 
\in G$ and $\cP$ is a property of representations in $\mathrm{Hom}(\Gamma, G)$, then we say that a \emph{generic $T$-constrained representation satisfies $\cP$}
if the subset $\{\rho \in \mathrm{Hom}_T(\Gamma, G) \mid \rho \text{ has }\cP\} \subset \mathrm{Hom}_T(\Gamma, G)$ is comeager. It follows from the Baire category theorem that if a generic $T$-constrained representation satisfies countably many properties $\cP_1, \cP_2, \dots$, then it satisfies all of these properties simultaneously. 
\end{no}



\subsection{Non-singular ergodic theory}\label{SubsecNonSing}

In order to define the target groups of our representations, we need some background from non-singular ergodic theory. A convenient reference is \cite{DanilenkoSilva2023}.

    \begin{no} Throughout this article, all measurable spaces are assumed to be standard Borel spaces. We refer to a measurable bijection between standard Borel spaces as a \emph{Borel isomorphism}; by \cite[(15.2) Corollary]{Kechris1995} the inverse of a Borel isomorphism is automatically measurable. If $(X, \cB)$ is a standard Borel space and $\mu$ denotes a $\sigma$-finite Borel measure  on $X$, then we refer to $(X, \cB, \mu)$ as a \emph{standard measure space}; if $\mu$ is moreover a probability measure, then it is called a \emph{standard probability space}.

\item From now on let $(X, \cB, \mu)$ (or $(X, \mu)$ for short) be a standard measure space. We denote by $[\mu]$ the \emph{measure class} of $\mu$, i.e.\ the collection of all  $\sigma$-finite Borel measures $\nu$ on $X$ such that $\nu(A) = 0 \iff \mu(A) = 0$ for all $A \in \cB$. Every measure class can be represented by a probability measure. Given $A, B \in \cB$ we define
\[
A \sim B \;:\Longleftrightarrow\; \mu(A \triangle B) = 0.
\]
The quotient of $\cB$ by this equivalence relation is a Polish group of exponent $2$, where $\triangle$ is the group operation. It is denoted $\malg(X,[\mu])$ and called the \emph{measure algebra} of $(X, [\mu])$.
\end{no}

\begin{no} \label{par:nonsingular on ma}
Let $(X, \cB, \mu)$ be a standard measure space. A Borel automorphism $T$ of $X$ is called \emph{non-singular} if $[T_*\mu] = [\mu]$, and two such automorphisms are called \emph{equivalent} if they coincide on a set of full measure. Equivalence classes of non-singular Borel automorphisms form a group $\mathrm{Aut}(X,[\mu])$ under composition of representatives, and this group acts on $\malg(X,[\mu])$ by $[T]([A]) \coloneqq  [T(A)]$. We will need the following consequence of the Radon--Nikodym theorem: 
\begin{lemma*}[{\cite[Lemma 4.2.1]{Cohn2013}}]
    If $T$ is non-singular, then 
    \[
    \forall\, \epsilon > 0\; \exists\, \delta > 0\; \forall A \in \cB:\; \mu(A) < \delta \implies \mu(TA) < \epsilon.
    \]
\end{lemma*}
\end{no}
\begin{remark*}
As is customary, we will often not distinguish notationally between non-singular Borel automorphisms and their equivalence classes in $\Aut(X, [\mu])$; or between elements of $\cB$ and their image in the measure algebra. For example, given $T \in \Aut(X, [\mu])$ and $A \in \malg(X,[\mu])$ we write $T(A)$ 
for the element of the measure algebra obtained by applying a representative of $T$ to a representative of $A$. This abuse of notation is uncritical as long as all the concepts involved are independent of representatives.
\end{remark*}
\begin{no} If $(X, \cB, \mu)$ is a standard probability space and $T \in \mathrm{Aut}(X, [\mu])$, then the triple $(X, [\mu], T)$ is called a \emph{non-singular dynamical system}. More generally, if $\Lambda$ is a group, then a homomorphism $\Lambda \to  \mathrm{Aut}(X, [\mu])$ is called a \emph{non-singular action}.

Clearly, if $[\mu]$ contains a $T$-invariant probability measure $\nu$, then $(X, [\mu], T)$ is non-singular; such a non-singular dynamical system is called \emph{probability-measure preserving} (p.m.p.).

We say that $T \in \mathrm{Aut}(X, [\mu])$ is \emph{periodic} if almost every $T$-orbit is finite, and 
\emph{aperiodic} if almost every $T$-orbit is infinite. We say that $T$ is \emph{ergodic} if for every $A \in \cB$ with $T^{-1}(A) = A$ we have $\mu(A) = 0$ or $\mu(X \setminus A) = 0$; then $(X, [\mu], T)$ is called an \emph{ergodic system} and $\mu$ is called \emph{$T$-ergodic}.
\begin{lemma*}\label{lem:many disjoint translates}
    If $T  \in \mathrm{Aut}(X, [\mu])$ is aperiodic and ergodic, then for every $A \subset X$ of positive measure and all $\epsilon > 0$ and $M > 0$ there exists a subset $A' \subset A$ of positive measure such that $A', TA', \dots, T^MA'$ are pairwise measurably disjoint, and $\mu(A' \cup \dots \cup T^M A') < \epsilon$.
\end{lemma*}
\begin{proof} Let $p>M$ be a prime. By the Lehrer--Weiss version of the non-singular Rokhlin lemma \cite[Theorem 2.9]{DanilenkoSilva2023} there exists a subset $B \subset X$ of positive measure such that $B, TB, \dots, T^{p-1}B$ (and hence $T^kB, \dots, T^{k+p-1}B$ for all $k \in \bZ$) are pairwise measurably disjoint and $A \subset B \cup TB \cup \dots \cup T^{p-1}B$. Choose $k \in \{0, \dots, p-1\}$ such that $A'' := T^k(B) \cap A$ has positive measure; then $A'', TA'', \dots, T^MA''$ are pairwise measurably disjoint, and by Lemma \ref{par:nonsingular on ma} we can choose $A' \subset A'' \subset A$ such that $\mu(A' \cup \dots \cup T^M A') < \epsilon$. 
\end{proof}

\begin{no}

The \emph{ratio set} of a nonsingular system $(X,[\mu],T)$ is the multiplicative subsemigroup $R(T) \subset \R_{\geq 0} \cup \{\infty\}$ consisting of all $r \in \R$ such that for every measurable subset $A\subset X$ of positive measure and every $\epsilon>0$ there exists a measurable subset $B\subset A$ of positive measure such that
\begin{equation*}
\abs{\frac{d T^{k} \mu}{d \mu}(x) - r} < \epsilon \quad \text{for all } x \in B. 
\end{equation*}
It turns out that $R(T)$ depends only on $[\mu]$.
\end{no}

\begin{no}\label{KriegerOE}
Two non-singular dynamical systems $(X, [\mu], T)$ and $(X', [\mu'], T')$ are \emph{orbit equivalent} if there exists a Borel isomorphism $f \colon X \to X'$ such that $[f_* \mu] = [\mu']$ and $\langle T \rangle x = \langle T' \rangle f(x)$ for almost every $x \in X$. Ergodic non-singular dynamical systems have been classified up to orbit equivalence by Dye and Krieger. A fundamental invariant of orbit equivalence is the so-called \emph{Krieger type}; here a non-singular, ergodic dynamical system $(X, [\mu], T)$ is said to have
\begin{itemize}
\item \emph{type $\I_n$}, with $1 \leq n < \infty$, if $\mu$ has precisely $n$ atoms. In particular $[\mu]$ contains the unique normalized counting measure on $n$ atoms preserved by $T$,
\item \emph{type $\I_\infty$} if $\mu$ has countably many atoms. In particular $[\mu]$ contains the unique counting measure on infinitely many atoms preserved by $T$,
    \item \emph{type $\II_1$} if there exists a non-atomic probability measure $\nu \in [\mu]$ preserved by $T$,
    \item \emph{type $\II_\infty$} if there exists a non-atomic $\sigma$-finite infinite measure $\nu \in [\mu]$ preserved by $T$.
\end{itemize}
In all of these cases, the given measure class contains an invariant measure, and hence $R(T) = \{1\}$. If the given measure class does \emph{not} contain an invariant measure, then we say that the system is of 
\begin{itemize}
\item \emph{type $\III_0$} if $R(T) = \{0,1,\infty\}$,
\item \emph{type $\III_{\lambda}$} if $R(T) = \{\lambda^n \ | \ n \in \Z\} \cup \{0,\infty\}$ for some $\lambda \in (0,1)$,
\item \emph{type $\III_{1}$} if $R(T) = [0, \infty]$.
\end{itemize} 
The \emph{Dye--Krieger} theorem then says that the orbit equivalence type of an ergodic $(X, [\mu], T)$ is uniquely determined by its Krieger type, except in the case of systems of type $\III_0$, and that the orbit equivalence class of a system of type $\III_0$ is determined by the isomorphism type of the associated Krieger flow. See \cite{KatznelsonWeiss1991} for details. We say that $T$ is of Krieger type $\I$ it is of Krieger type $\I_n$ for some $n \in \bN \cup \{\infty\}$ and define Krieger types $\II$ and $\III$ similarly.
\end{no}

\begin{no}
    Concrete models for systems of different Krieger types are given by \emph{odometers}. Let $\Z_p \cong \{0,\dots,p-1\}^\bN$ be the $p$-adic integers, and let $T \colon \Z_p \to \Z_p, x\mapsto x+1$ be the $p$-adic odometer (adding machine).
    Then $(\Z_p,[\mu],T)$ is of type $\II_1$ for the Haar measure.
    If $p=2$ and $\mu = \bigotimes_{n \in \bN} (\frac{1}{1+\lambda} \delta_0 + \frac{\lambda}{1+\lambda} \delta_1)$ with $0 < \lambda < 1$, then $(\Z_2,[\mu],T)$ is of type $\III_\lambda$.
    Krieger types $\II_\infty, \III_0$ and $\III_1$ can be achieved with a more general odometer construction, see \cite[Example 5.1, Example 5.3]{DanilenkoSilva2023}
\end{no}

\begin{no}
\begin{definition*}\label{DefWW} Let $T \in \Aut(X, [\mu])$. Then $W \in \cB$ is called \emph{weakly wandering} for $T$ if there exist $0 = k_0 < k_1 < \dots$ such that
$\mu(T^{k_i}W \cap T^{k_j} W)= 0$ for all $i \neq j$.
\end{definition*}
\begin{proposition*}[{\cite{HajianKakutani1964}}]\label{WWTypes1} For a non-singular, ergodic dynamical system $(X, [\mu], T)$ the following are equivalent:
\begin{enumerate}[(i)]
\item There is no $T$-invariant probability measure in $[\mu]$.
\item $T$ is of Krieger type $\I_\infty, \II_\infty$ or $\III_\lambda$ for some $\lambda \in [0,1]$.
\item There is a weakly wandering set of positive measure for $T$.
\end{enumerate}
\end{proposition*}
See Theorem \ref{ThmKrengel} below for yet another characterization.
\end{no}
\begin{no} Let $(X, \mu)$ be a standard probability space. An equivalence relation $\cR$ on $X$ is called \emph{measurable} if it is a $\mu \times \mu$-measurable subset of $X \times X$. 
We say that a measurable equivalence relation $\cR$ is \emph{countable} (respectively \emph{finite}) if almost every equivalence class of $\cR$ is countable (respectively finite). It is called \emph{hyperfinite} if $\cR = \bigcup_{n \in \bN} \cR_n$ for finite measurable subequivalence relations $\cR_n$ and \emph{ergodic} iff for every $A \in \cB$ the $\cR$-saturation $\bigcup_{x \in A} [x]$ has measure $0$ or $1$. 
\begin{example*} Every non-singular dynamical system $(X, [\mu], T)$ defines a hyperfinite measurable equivalence relation $\cR_T \coloneqq  \{(x, y) \in X \times X \mid y \in \langle T \rangle x\}$ called the \emph{orbit equivalence relation} of $T$. Given a hyperfinite equivalence relation $\cR$ on $(X, [\mu])$ we say that $T \in \Aut(X, [\mu])$ is a \emph{generator} for $\cR$ if $\cR = \cR_T$ almost everywhere. By the Ornstein--Weiss theorem, ergodic hyperfinite equivalence relations are precisely the orbit equivalence relations of ergodic non-singular transformations (up to almost everywhere coincidence). In particular, $T$
is ergodic if and only if $\cR_T$ is ergodic.

Moreover, $T$ is aperiodic if and only if almost every class of $\cR_T$ is infinite; we then also say that $\cR_T$ is aperiodic.
Moreover, the Krieger type of $T$ depends only on $\cR_T$ and is thus also called the \emph{Krieger type} of $\cR_T$. Then an \emph{ergodic} hyperfinite equivalence relation $\cR$ on $(X, [\mu])$ is aperiodic
 unless it is of Krieger type $\I_n$. In particular, this is the case if $\cR$ is ergodic and $\mu$ is non-atomic (i.e. $\cR$ is not of Krieger type $\I$).
\end{example*}
\end{no}


\subsection{Measurable full groups}\label{SubsecMFG}
\begin{no}\label{MFG}
We are now in the position to define the target groups of our representations; these are measurable full groups of non-singular dynamical systems or, equivalently, the associated measurable equivalence relations. Throughout this subsection let $\cR$ be a hyperfinite ergodic equivalence relation on a standard probability space $(X, [\mu])$.
\end{no}
\begin{definition*} The \emph{measurable full group} of $\cR$ is the subgroup
\[
\mfg = \{g \in \Aut(X, [\mu]) \mid (x, g(x)) \in \cR \text{ for almost all }x \in X\} < \Aut(X, [\mu]).
\]
\end{definition*}
\begin{no} \label{ssec:uniform} \label{ActionCocycle}
By definition, if $T$ is a generator of $\cR$, then $g \in [\cR]$ if and only if almost every $g$-orbit is contained in the corresponding $T$-orbit. Thus there then exists a (non-strict) cocycle $\alpha: [\cR] \times X \to \mathrm{Sym}(\bZ)$, $(g,x) \mapsto \alpha_x(g)$, unique up to almost everywhere equality, such that for all $g \in [\cR]$ and almost all $x \in X$ we have 
\[
g(T^n x) = T^{\alpha_x(g)(n)}(x).
\]
We refer to $\alpha$ as the \emph{action cocycle} of $[\cR] \curvearrowright X$ with respect to $T$.

On $[\cR]$ we define the \emph{uniform metric} defined by $d(g,h) = \mu(\{x \in X \mid g(x) \neq h(x)\})$ and equip $[\cR]$ with the corresponding topology; then $[\cR]$ is a Polish group and the Krieger type of $\cR$ (or $T$) depends only on the isomorphism class of $[\cR]$ as a topological group. If $\cR$ is of type $\I_n$ or $\II_1$, then the uniform metric is complete, otherwise a complete metric for $[\cR]$ is given by $d'(g,h) = d(g,h) + d(g^{-1},h^{-1})$. In the sequel we always equip $[\cR]$ with the uniform metric, and spaces of (constrained) representations into $[\cR]$ with the corresponding metric as in \S \ref{Ev}.
 
 If $\cR$ is of type $\I_n$ with $n < \infty$, then $[\cR] \cong \mathrm{Sym}_n$ is finite, and if $\cR$ is of type $\I_\infty$ then $[\cR]$ is isomorphic to the group $\Sym(\Z)$ of permutations on $\bZ$ with the permutation topology. Otherwise, $[\cR]$ is simple \cite{Eigen1981} and contractible \cite{Danilenko1995}.
 \end{no}
\begin{no} In order to define an element of the full group we will usually only specify it on a conull set; this is possible due to the following proposition. Here, by a \emph{partial automorphism} of $(X, [\mu])$ we mean a Borel isomorphism $g_o\colon \dom(g_0) \to \ran(g_o)$ between Borel subsets $\dom(g_o), \ran(g_o) \subset X$. 
\begin{proposition*} Let $g_o$ be a partial automorphism of $(X, [\mu])$ such that $\dom(g_o)$ and $\ran(g_o)$ are conull in $X$. Then the following are equivalent:
\begin{enumerate}[(i)]
\item There exists $g \in [\cR]$ whose representatives agree with $g_o$ almost everywhere.
\item $g_o(x) \in [x]_{\cR}$ for almost all $x \in \dom(g_o)$.
\end{enumerate}
In this case, $g$ is uniquely determined by $g_o$.
\end{proposition*}
\begin{proof} It is clear that (i)$\implies$(ii) and that $g$ is uniquely determined by $g_o$. For the implication (ii)$\implies$(i) assume that $\dom(g_o) = X \setminus N_1$ and $\ran(g_o) = X \setminus N_2$; we can enlarge $N_1$ and $N_2$ and may thereby assume without loss of generality that $g_o(x) \in [x]_{\cR}$ for all $x \in \dom(g_o)$ and that $N_1$, $N_2$ and $g_o$ are Borel. Then for $i \in \{1,2\}$ the saturation $M_i \coloneqq  \bigcup_{n \in \Z} T^n N_i$ is a Borel nullset, and hence $M \coloneqq  M_1  \cup M_2$ is a Borel nullset. We claim that 
\begin{equation}\label{ClaimExtension}
X \setminus M = \{x \in X \mid g_o|_{[x]} \colon [x] \to [x] \text{ well-defined and bijective}\}.
\end{equation}
Indeed, the restriction of $g_o$ to $[x]$ is well-defined if and only if $[x] \cap N_1 = \emptyset$, which is the case if and only if $x \notin M_1$, and in this case $g_o|_{[x]}$ is injective. If $x \notin M_1$, then $g_o|_{[x]}$ is surjective if and only if  $g_o([x]) = [g_o(x)]$, i.e.\ $g_o(x) \notin M_2$, hence \eqref{ClaimExtension} follows.

By \eqref{ClaimExtension} we have a well-defined restriction $g_1 \coloneqq  g_o|_{X \setminus M}\colon X \setminus M \to X \setminus M$, and since $g_o$ is bijective and Borel, also $g_1$ is a Borel automorphism of $X\setminus M$. Now extend $g_1$ by the identity on $M$; this defines a class $g \in [\cR]$ with the desired properties.
\end{proof}
In other words, to specify an element of $[\cR]$ it is enough to specify it on a conull set and to ensure that it preserves the equivalence relation there. In some cases we can also extend partial isomorphisms which are defined on a smaller subset:
\begin{lemma*}[Extension lemma]\label{lem:extension_infty}
Assume that $\cR$ is non-atomic and let $g'$ be a partial automorphism of $(X, [\mu])$ such that $g'(x) \in [x]_{\cR}$ for almost all $x \in \dom(g')$. If $0 < \mu(\dom(g')), \mu(\ran(g')) < 1$, then there exists $g \in [\cR]$ with $g|_{\dom(g)} = g'$.
\end{lemma*}
\begin{proof} Since $\cR$ is ergodic and non-atomic, $\cR$ is either of type $\II$ or of type $\III$. In the former case, apply \cite[Lemma 7.16]{LeMaitre2024}, and in the latter case apply \cite[Prop.~11]{KaichouhLeMaitre2015}.
\end{proof}
\end{no}

\begin{no}
    An important fact is that the set of \emph{periodic} elements is a dense $G_\delta$ subset of $[\cR]$. Here we call an element $g \in [\cR]$ periodic if almost every point is periodic, i.e. for almost every $x \in X$ there exists $n \geq 1$ with $g^n(x)=1$.
\end{no}

\subsection{Conservative transformations and e.c.\ actions}\label{SubsecConservative}
\begin{no}
We have the following strengthening of the notion of a weakly wandering set:
\begin{definition*} Let $T \in \ns$. A set $W \in \cB$ is called a \emph{wandering set for $T$} if $\mu(T^n(W) \cap W) = 0$ for all $n \in \Z \setminus \{0\}$.
\end{definition*}
\begin{proposition*}[{\cite[Prop.~2.1]{DanilenkoSilva2023}}]\label{ConservativeAll} Let $(X, [\mu], T)$ be a nonsingular dynamical system. Then the following are equivalent:
\begin{enumerate}[(i)]
\item Every wandering set for $T$ is a nullset.
\item For every $A \in \cB$ with $\mu(A)>0$ there exists $n>0$ such that $\mu(A \cap T^{-n}A) > 0$.
\item For every $A \in \cB$,
\[
\mu\left(A \setminus\bigcup_{n=1}^\infty T^{-n}A \right)= 0.
\]
\item $T$ is \emph{incompressible}, i.e.\ if $C \in \cB$ with $T^{-1}C \subset C$, then $\mu(C \setminus T^{-1}C) = 0$.
\end{enumerate}
\end{proposition*}
\begin{definition*} If a system  $(X, [\mu], T)$ satisfies the equivalent conditions from Proposition \ref{ConservativeAll}, then it is called \emph{conservative}, and $T$ is called a \emph{conservative transformation}.
A non-singular action of a group $\Lambda$ on a standard probability space $(X, \mu)$ is \emph{elementwise conservative} (e.c.) if every element of $\Lambda$ acts by a conservative transformation.
\end{definition*}

\begin{remark*} Every nonsingular dynamical system $(X, [\mu], T)$ admits a unique, up to nullsets, decomposition $X = C \sqcup D$ into $T$-invariant Borel sets (called \emph{Hopf decomposition}) such that $D = \bigsqcup_{n \in \Z} T^n W$ for some wandering set $W \subset X$ and such that $T|_C$ is conservative (\cite[Thm.~2.2]{DanilenkoSilva2023}). The set $C$ is then called the \emph{conservative part} of $X$.
\end{remark*}
\end{no}

\begin{no}
\begin{example*}\label{ConservativeExample} Let $(X, [\mu], T)$ be a non-singular system.
\begin{enumerate}[(i)]
\item If $T$ is \emph{periodic}, in the sense that almost every $T$-orbit is finite, then $T$ is conservative.
\item If $\mu$ is non-atomic and $T$ is ergodic, then $T$ is conservative. In particular, if $T$ is ergodic and of Krieger type $\II_\infty$ or $\III_\lambda$ for some $\lambda \in [0,1]$, then $T$ admits a weakly wandering set of positive measure by Proposition \ref{WWTypes1}, but no wandering set of positive measure by Proposition \ref{ConservativeAll}.
\item If $X$ admits an (essentially) disjoint decomposition $X = \bigsqcup X_i$ into finitely many $T$-invariant pieces $X_i$, then $T$ is conservative if each each of the restrictions $T|_{X_i}$ is conservative.
\item It follows from (i)-(iii) that if $\mu$ is non-atomic and $X$ decomposes into a part on which $T$ acts periodically and finitely many subsets on which $T$ acts ergodically, then $T$ is conservative.
\end{enumerate}
\end{example*}
\begin{remark*} It $\mu$ has atoms, then a non-singular ergodic transformation of $(X, [\mu])$ need not be conservative. Namely, if $T$ denotes the shift on $\bZ$ and $\mu$ is equivalent to the counting measure, then $T$ is ergodic, but every singleton is a wandering set for $T$. 
This is essentially the only example of an ergodic transformation that is not conservative.
However, if $T$ is \emph{recurrent} in the sense that for every $A \in \cB$ with $\mu(A)>0$ and almost every $x \in X$ there exists $n > 0$ with $T^nx \in A$, then it is always conservative (and ergodic). By Poincar\'e recurrence, this is the case for any probability measure preserving ergodic transformation. Conversely, every ergodic and conservative transformation is recurrent.
\end{remark*}
\end{no}

\begin{no}
We record the following consequence of Example \ref{ConservativeExample}.(iv) for later use:
\begin{proposition*}[Conservativity criterion]\label{Conservativity} Assume that $(X, [\mu], T)$ is an ergodic non-singular system and that $\mu$ is non-atomic. If $g \in [\cR_T]$ and there exists $k \in \mathbb{N}$ such that
almost every $T$-orbit contains at most $k$ infinite $g$-orbits, then $g$ is conservative.
\end{proposition*}
\begin{proof} Choose $k$ such that every $T$-orbit contains at most $k$ infinite $g$-orbits and partition $X = X_p \sqcup X_{np}$ into $g$-invariant measurable sets with $X_p$ containing all the finite $g$-orbits and $X_{np}$ containing all the infinite ones. By Example \ref{ConservativeExample}.(i) we may assume that $\mu(X_{np}) > 0$. 

Assume that $X_{np} = X_1 \sqcup X_2 \sqcup \ldots \sqcup X_L$ is a partition into $g$-invariant sets of positive measure. By ergodicity of $T$, almost every $T$-orbit $\cO$ intersects every $X_i$, and the intersection $\cO \cap X_i$ then contains at least one $g$-orbit $\cO_i$, which is infinite by construction of $X_{np}$. Then $\cO_1, \dots, \cO_L$ are infinite $g$-orbits in $\cO$, and thus $L \leq k$ by assumption. This implies that we can decompose $X_{np}$ into at most $k$ subsets on which $g$ acts ergodically, and hence the proposition follows from Example \ref{ConservativeExample}.(iv).
\end{proof}
\end{no}

\subsection{The Chabauty space and boomerang subgroups}\label{sect:chabauty}

\begin{no}
If $H$ is a locally compact second countable group, then the space $\Sub(H)$ of closed subgroups of $H$ carries a canonical compact metrizable topology known as the \emph{Chabauty--Fell topology}, which is invariant under the $H$-action by conjugation and hence gives rise to a topological dynamical system $H \curvearrowright \Sub(H)$. We will only need this topology in the case of a countable discrete group $\Lambda$, in which case it coincides with the restriction of the product topology via the embedding $\Sub(\Lambda) \subset \{0,1\}^\Lambda$, which can be shown to have closed image. In particular, for a countable discrete group $\Lambda$ the \emph{Chabauty space} $\Sub(\Lambda)$ is totally disconnected.

If $\Lambda \curvearrowright X$ is a Borel action on a standard Borel space, then the stabilizer map $\mathrm{Stab}\colon X \to \Sub(\Lambda)$ is Borel. In particular, if the stabilizer map is injective, then it is a Borel isomorphism onto its image. More generally, if $(X, \mu)$ is a standard probability space and $\Lambda < \Aut(X,[\mu])$, then we say that $\Lambda$ is \emph{totally-non-free} if there is a full measure subset $X' \subset X$ such that $\Stab \colon X' \to \Sub(H)$ is injective; this then implies that $\mathrm{Stab}$ induces an isomorphism $(X, \mu) \to (\Sub(\Lambda), \Stab_*\mu)$.

\end{no}
\begin{no} From now on $\Lambda$ denotes a countable discrete group. We recall from the introduction that a subgroup $\Delta \in \mathrm{Sub}(\Lambda)$ is a boomerang subgroup if for every $\gamma \in \Lambda$ there exists a sequence $n_k \to \infty$ with $\gamma^{n_k} \Delta \gamma^{-n_k} \to \Delta$ in $\mathrm{Sub}(\Lambda)$. Elementary examples of boomerang subgroups are
\begin{enumerate}
    \item normal subgroups,
    \item finite index subgroups,
    \item more generally, subgroups whose normalizer has finite index,
    \item more generally, subgroups such that every group element has a power in the normalizer.
\end{enumerate}
All of these examples of boomerang subgroups $\Delta < \Lambda$ have in common that for every $\gamma \in \Lambda$, the $\langle \gamma \rangle$-orbit of $\Delta$ is finite. In fact it is not hard to show that this is necessarily the case whenever $\Delta$ is finitely generated, or whenever $\Delta$ is not in the perfect kernel of $\Sub(\Lambda)$.
\end{no}
\begin{no}
The topology of $\Sub(\Lambda)$ admits the following equivalent description, which is more convenient for our purposes. For any finite subset $F \subset  \Lambda$, and any $\Delta \in \Sub(\Lambda)$, we denote
    \begin{enumerate}
        \item $\Env(F) \coloneqq  \{H \in \Sub(\Lambda) \mid F \subset H \}$
        \item $\Miss(F) \coloneqq  \{H \in \Sub(\Lambda) \mid F \cap H = \emptyset \}$
        \item $\Nbh(\Delta,F) \coloneqq  \{H \in \Sub(\Lambda) \mid F \cap H = F \cap \Delta\} = \Env(F \cap \Delta) \cap \Miss(F \setminus \Delta)$.
    \end{enumerate}
    Then, as $F$ varies over all finite subsets of $\Lambda$
    the sets $\Env(F)$ and $\Miss(F)$ form a basis for the topology on $\Sub(\Lambda)$ and the sets $\Nbh(\Delta,F)$ constitute a neighborhood basis of $\Delta$ in $\Sub(\Lambda)$. This implies that $\Delta \subset \Lambda$ is a boomerang if and only if for every $\gamma \in \Lambda$ and every finite subset $F \subset \Lambda$ there exists $n > 0$ with $\gamma^n \Delta \gamma^{-n} \in \Nbh(\Delta,F)$, i.e.
\[
\gamma^n \Delta \gamma^{-n} \cap F = \Delta \cap F.
\]
Using this terminology we prove the following lemma promised in Remark \ref{rem:golb_conservativity}, which can be seen as a consequence of the ``locally essential lemma'' (\cite[Lemma 2.3]{Gla:lin_irs})
\begin{lemma*} \label{lem:whole_gp}
Let $\Lambda$ be a countable group. Then 
the conjugation action $\Lambda \curvearrowright \Sub(\Lambda)$ is globally conservative with respect to any non-singular probability measure $\mu$ on $\Sub(\Lambda)$.  
\end{lemma*}
\begin{proof} Let $B \subset \Sub(\Lambda)$ with $\mu(B) > 0$. We need to show that there exists $\gamma \in \Lambda$ such that $\mu(\gamma B \gamma^{-1} \cap B) > 0$. Write $\mu = 
\lambda \delta_{\{1\}} + (1-\lambda) \mu'$ with $\mu'(\{\{1\}\}) = 0$ and $\lambda \in [0,1]$. The statement is obvious if $\mu'(B) = 0$, hence we assume that $\mu'(B) > 0$.

By definition, $\mu'$ is supported on the countable union $\bigcup_{\gamma \ne 1} \Env(\{\gamma\})$, hence we can choose $\gamma \in \Lambda \setminus \{1\}$ such that $B' := \Env(\{\gamma\}) \cap B$ satisfies $\mu'(B') > 0$. Now $\gamma$ acts trivially on $\Env(\{\gamma\})$, hence on $B'$, and thus $\mu(\gamma B \gamma^{-1} \cap B)  \geq \mu'(\gamma B \gamma^{-1} \cap B) \geq \mu'(B') > 0$.
\end{proof}
\end{no}

\section{Basics on representations into measurable full groups}
\label{sect:basics}

In this section $\Lambda$ is any countable group,
$\cR$ a hyperfinite Borel equivalence relation on $(X,\cB,[\mu])$ with some generator $T \in \Aut(X,[\mu])$ and $G = [\cR]$ denotes the corresponding measurable full group.

\subsection{$G_\delta$-properties of generic representations}

\begin{no}
Throughout this article we will establish that various sets of (constrained) representations are dense $G_\delta$-sets. The $G_\delta$-property is usually easy to establish, and one can use the following general principle:
\begin{proposition*} \label{prop:G_del}
If $\Lambda$ is a countable group and $P \subset \Sub(\Lambda)$ is $G_{\delta}$, then also
$$\Xi(P) \coloneqq  \{\rho \in \Hom(\Lambda,G) \ | \ \mu(\{x \ | \ \Lambda_x \in P\}) = 1\} \subset \Hom(\Lambda,G)$$ is $G_\delta$,
where $\Lambda_x$ is the stabilizer of $x$ in $\Lambda$ with respect to the action induced by $\rho$.
\end{proposition*}
\begin{proof}
It suffices to show that for every open subset $P \subset \Sub(\Lambda)$ and every $\epsilon > 0$ the set  
$$\Xi(P,\epsilon) \coloneqq  \{\rho \in \Hom(\Lambda,G) \ | \ \mu(\{x \ | \ \Lambda_x \in P\}) > 1 - \epsilon \}$$
is open. Indeed, if $P$ is $G_\delta$ then we can realize it as a descending intersection of open sets $P = \bigcap_{i \in \bN}P_i$ and $\Xi(P) =  \bigcap_{n,i \in \bN} \Xi(P_i,1/n)$.  Hence we assume that $P$ is open and fix $\rho \in \Xi(P,\epsilon)$. 

Let $\Sigma \in \Sub(\Lambda)$ and express $\Lambda = \bigcup B_n$ as an ascending union of finite sets $B_n$. 
Then the sets $\Nbh(\Sigma,B_n) = \{\Delta \in \Sub(\Lambda) \ | \ \Sigma \cap B_n = \Delta \cap B_n\}$
 form a local basis of Chabauty open neighborhoods around $\Sigma$. 
 
 Now let $x \in X$. If $\Lambda_x \not \in P$ we set $n(x,\rho) \coloneqq \infty$. Otherwise (using that $P$ is open) there exists some 
$n = n(\rho,x)$ such that 
$\Lambda_x \in \Nbh(\Lambda_x,B_n) \subset P$. By definition of $\Xi(P,\epsilon)$ we find $0 < \eta = \eta(\rho) < \epsilon/2$ such that
$n(x,\rho) = \infty$ on a set of measure at most $\epsilon - 2 \eta$. Thus
we can fix one global $N = N(\rho) \in \bN$ such that 
$$\mu \left(B(\rho,N)\right) = 1 - \epsilon + \eta > 1-\epsilon, \qquad {\text{where }} B(\rho,N) = \{x \in X \ | \ U(\Lambda_x,B_N) \subset P\}.$$
Now choose an open neighborhood $\rho \in O \subset \Hom(\Lambda,G)$ small enough to satisfy
$$\mu \left(\{x \in X \mid \forall \rho' \in O\, \forall \gamma \in B_N: \; \rho(\gamma)x \ne \rho'(\gamma)x \ \} \right) < \frac{\eta}{\abs{B_N}}.$$
Thus, for any $\rho' \in O$ we have
\begin{eqnarray*}
\mu \left(B(\rho',N(\rho'))\right) & \ge & \mu \left(B \left(\rho,N(\rho)\right) \cup \bigcup_{\gamma \in B_N} \{x \ | \ \rho(\gamma)x \ne \rho'(\gamma)x \} \right) \\
& > &  1-\epsilon + \eta - \sum_{\gamma \in B_N} \frac{\eta}{\abs{B_N}} > 1-\epsilon,
\end{eqnarray*}
hence $\rho \in O \subset \Xi(P, \eps)$.
\end{proof}
If we specialize to the case of a countable non-abelian free group $\Gamma$ and use the fact that $\Hom_T(\Gamma,G)$ is a closed, and hence $G_{\delta}$, subset of $\Hom(\Gamma,G)$, then we obtain:
\begin{corollary*} \label{cor:G_del}
If $P \subset \Sub(\Gamma)$ is $G_{\delta}$ then 
$$\Xi_T(P) \coloneqq  \Hom_T(\Gamma,G) \cap \Xi(P) = \{\rho \in \Hom_T(\Gamma,G) \ | \ \mu\{x \ | \ \rho(\Gamma)_x \in P\} = 1\},$$
is a $G_{\delta}$-subset of $\Hom_T(\Gamma,G)$. 
\end{corollary*}
\end{no}

\begin{no}
    The following proposition shows that various properties of interest to us are $G_\delta$ in the Chabauty space. Here and in the sequel we use the symbol $\subset_f$ to denote finite subsets; given a subgroup $\Delta < \Lambda$ we denote by $\pi_\Delta\colon \Lambda \to \Lambda/\Delta$ the canonical quotient map.
\begin{proposition*}\label{prop:Gdelta properties in the Chabauty}
If $\Lambda$ is a countable group, then the sets $\Sub^{\text{cf}}(\Lambda)$, $\Sub^{\text{coa}}(\Lambda)$ and $\Sub^{\text{cht}}(\Lambda)$ of core free subgroups, co-amenable subgroups and co-highly transitive subgroups respectively are
$G_{\delta}$-subsets of $\mathrm{Sub}(\Lambda)$.    
\end{proposition*}

\begin{proof} Firstly, we can write 
$$\Sub^{\text{cf}}(\Lambda) = \bigcap_{1 \ne \gamma \in \Lambda} \left( \bigcup_{\eta \in \Lambda} \Miss(\{\eta \gamma \eta^{-1}\}) \right),$$
which is a countable intersection of Chabauty-open subsets of $\mathrm{Sub}(\Lambda)$.    

Secondly, given $A, F \subset_f \Lambda$ and $\epsilon > 0$ we denote by $\Omega_{A,F,\epsilon}$ the set of $\Delta \in \Sub(\Lambda)$ such that $\pi_\Delta \colon \Lambda \to \Lambda/\Delta$ is injective on $A$ and $\pi_\Delta(A) \subset \Lambda/\Delta$ is an $(F,\epsilon)$-F{\o}lner subset. Explicitly, the latter means that for all $\gamma \in \Lambda$ we have
$|\gamma \pi_\Delta(A) \triangle \pi_\Delta(A)|/| \pi_\Delta(A)| < \eps$, and by definition we have
\[
\Sub^{\text{coa}}(\Lambda) = \bigcap_{n \geq 1} \bigcap_{F \subset_f \Lambda} \bigcup_{A \subset_f \Lambda} \Omega_{A,F,1/n}.
\]
We thus fix $A,F \subset_f \Lambda$ and $\eps > 0$; we have to show that $\Omega_{A,F,\epsilon}$ is open.
Concerning the first condition, we note that
\[
\{\Delta \in \Sub(\Lambda) \mid \pi_\Delta|_{A} \text{ injective} \} = \Miss(A^{-1} A \setminus \{1\}),
\]
and for every $\Delta$ is in this set and all $\gamma \in \Lambda$ we have $|A|=|A/\Delta|= |\gamma A/\Delta|$.
Next, take any $\Delta \in \Sub(\Lambda)$ any $\gamma \in F, \omega,\omega' \in A$ and observe that $\gamma \omega' \Delta = \omega \Delta  \in \gamma A/\Delta \cap A/\Delta$ if and only if $\omega^{-1} \gamma \omega' \in \Delta$. So we see that
if $\Delta \in \Omega_{A,F,\epsilon}$ then also
\[
\Nbh(\Delta,A^{-1} F A \cup A^{-1} A) \subset \Omega_{A,F,\epsilon}.
\]
This shows that the set of co-amenable subgroups is $G_\delta$.

Finally, given $\gamma \in \Lambda$, $\Omega \subset_f \Lambda$ and $\sigma$ a permutation of $\Omega$ we set 
\[
    H(\gamma,\Omega,\sigma) \coloneqq \{\Delta \in \Sub(\Lambda) \ | \forall \omega \ne \omega' \in \Omega \colon \, \omega^{-1} \omega' \not \in \Delta, \, {\text{ and }} \forall \omega \in \Omega \colon \, \gamma \omega \Delta = \sigma(\omega) \Delta \}. 
\]
Then $\Delta \in H(\gamma, \Omega, \sigma)$ if and only if $\pi_\Delta$ is injective on $\Omega$ and $\gamma$ induces the permutation $\sigma$ on $\pi_\Delta(\Omega)$, and $H(\gamma,\Omega,\sigma)$ is Chabauty-open, since
\[  H(\gamma,\Omega,\sigma)= \Miss(\Omega \setminus \{1\}) \cap \bigcap_{\omega \in \Omega} \Env(\{\sigma(\omega)^{-1} \gamma \omega\})\]
The proposition then follows from the fact that 
\[\Sub^{\text{cht}}(\Lambda) = \bigcap_{\Omega \subset_f \Lambda} \left( \bigcap_{\Omega' \subset \Omega} \bigcap_{\sigma \in \Sym(\Omega')} \bigcup_{\gamma \in \Lambda}  H(\Lambda,\gamma,\Omega',\sigma) \right).\qedhere\]
\end{proof}
Combining this with Corollary \ref{prop:G_del} we deduce:
\begin{corollary*}
    The subsets of $\Hom_T(\Gamma, G)$ consisting of representations $\rho$ such that for almost every $x \in X$ the action of $\Gamma$ on the orbit $\rho(\Gamma)x$ is amenable, faithful or highly-transitive respectively are $G_\delta$ in $\Hom_T(\Gamma, G)$.
\end{corollary*}
\end{no}

\subsection{Faithful representations}

\begin{no} 
Given a representation $\rho\in \Hom(\Lambda,G)$ we say that $\rho$ is \emph{faithful on almost every orbit} if and only if 
\[
\mu(\{x \in X \mid \forall\, \gamma \in \Lambda \setminus\{1\}: \rho(\gamma)|_{[x]} \neq 1\}) = 1.
\]
This implies that $\rho$ is faithful, and if $\cR$ is ergodic, then also the converse holds:
    \begin{lemma*}\label{lem:faithful equals on every orbit}
        Assume that $\cR$ is ergodic and let $\rho \in \Hom(\Lambda,G)$. Then $\rho$ is faithful if and only if it is faithful on almost every orbit.
    \end{lemma*}

    \begin{proof} For every $\gamma \in \Lambda$ the set $X_\gamma \coloneqq \{x \in X \mid \rho(\gamma)|_{[x]} \neq 1\}$ is a measurable union of $\cR$-orbits, hence has measure either $0$ or $1$. Thus
    \[
    \rho(\gamma) \neq 1 \iff \mu(X_\gamma) \neq 0 \iff \mu(X_\gamma) = 1.
    \]
    Now assume that $\rho$ is faithful; since $\Lambda$ is countable we deduce that
    \[
    \forall \gamma \in \Lambda \setminus \{1\}:\, \mu(X_\gamma) = 1 \Longrightarrow \mu\left(\bigcap_{\gamma \in \Lambda \setminus\{1\}} X_\gamma \right) = 1 \Longrightarrow \rho \text{ faithful on every orbit.}\qedhere
    \]
    \end{proof}
\end{no}

\section{Warmup: The case of type $\I_\infty$}\label{SecIInf}

\begin{no}\label{WarmupLemma} In this section we sketch a proof of Theorem \ref{thm:S_infty}. Thus let $G \coloneqq \Sym(\bZ)$ denote the group of all permutations of $\Z$, i.e.\ bijective maps $g\colon \Z \to \Z$, and let $\sigma \in G$ denote the shift map given by $\sigma(n) \coloneqq n+1$. Given $g \in \mathrm{Sym}(\bZ)$ we define
\[
\mathrm{supp}(g) := \{n\in \bZ \mid g(n) \neq n\}.
\]
Note that in general the support of a permutation is allowed to be infinite. 

Recall that the notation $F \subset_f \bZ$ means that $F$ is a finite subset of $\bZ$. In this case, every injection $F \hookrightarrow \bZ$ can be extended to a permutation in $\mathrm{Sym}(\bZ)$.

We consider constrained representations $\rho \in \Hom_\sigma(\Gamma, G)$. An open neighborhood basis of such a representation $\rho$ is given by
\[U_{F,N}(\rho) \coloneqq  \{\eta \in \Hom_\sigma(\Gamma,G) \mid \eta(b_i)|_F=\rho(b_i)|_F  \text{ for all } i = 1,\dots,N \},\]
where $F$ runs over all finite subsets of $\bZ$ and $N$ runs over all positive integers which are strictly smaller than the rank of $\Gamma$. When dealing with integers we will use the following interval notation: Given
integers $a \leq b$ we denote by
\[
[a,b] := \{a, a+1, \dots, b-1, b\}
\]
the integer points in the corresponding closed interval.
In the following lemma
given $g \in G$ we denote by $C_G(g)$ the centralizer of $g$ in $G$.
  \begin{lemma*}
    Let $g \in G$ and $O \subset G$ a neighborhood of $g$. 
    \begin{enumerate}
        \item There exists $k \in \bN$ such that for all $k' > k$ there exists $h \in O \cap C_G(\sigma^{k'})$ with $h((-k',k']) = (-k',k']$.
        \item For every $k \in \bN$, the set $\bigcup_{k' \geq k} C_{G}(\sigma^{k'})$ is dense in $G$.
        The subset $\{\rho \in \Hom_\sigma(\Gamma, G) \mid \exists n \geq k \colon \rho(\Gamma) \leq C_{G}(\sigma^n)\} \subset \Hom_\sigma(\Gamma,G)$ is dense.
    \end{enumerate}
    \end{lemma*}

    \begin{proof} (1) In the first step we can approximate $g$ to obtain a finitely supported element $h'$ with support contained in $[-k,k]$.
    Then, we extend $h'$ periodically via $h'|_{(2nk'-k',2nk'+k']} \coloneqq \sigma^{2nk'}h'|_{(-k',k']} \sigma^{-2nk'}$.

    (2) The first statement follows immediately from (1).
    For the second statement, recall that $\bigcup_{r \geq 1} \Hom_\sigma(\Gamma_r,G)$ is dense in $\Hom_\sigma(\Gamma,G)$, so we can assume that $\Gamma=\Gamma_r$ is finitely generated.
    Let $\rho \in \Hom_\sigma(\Gamma, G)$.
    For every $r > 0$ we can find $h_1,\dots,h_{r-1}$ arbitrarily close to $\rho(b_1),\dots,\rho(b_{r-1})$
    and $k'$ with $h_1,\dots,h_{r-1} \in C_G(\sigma^{k'})$.
    Now $\eta \in \Hom_\sigma(\Gamma,G)$ with $\eta(b_i) \coloneqq h_i$ satisfies $\eta(\Gamma) \subset C_G(\sigma^{k'})$.
    \end{proof}

    
\end{no}
\begin{no}
\begin{proof}[Sketch of proof of Theorem \ref{thm:S_infty}] 

\item (\ref{itm:atomic finorb}) If $w$ is conjugate to $a^k$ for some $k \in \bZ \setminus\{0\}$, then $\rho(w)$ is conjugate to $\sigma^k$ hence not conservative and in particular does not have finite orbits. It thus remains to show that (i)$\Longrightarrow$(ii). For this we observe that if we abbreviate $\Gamma_0 \coloneqq  \Gamma \setminus \bigcup_{k \in \bZ} \bigcup_{\gamma \in \Gamma} \gamma a^k \gamma^{-1}$, then
\begin{align*}
    \{\rho \in \Hom_\sigma(\Gamma,G) \mid \forall w \in \Gamma_0 \, \forall x \in \Z \, \exists n > 0 \colon \rho(w)^n x = x \} &= \bigcap_{w \in \Gamma_0} \bigcap_{x\in \Z} \Omega_{x,w},
\end{align*}
where $\Omega_{x,w} = \{\rho \mid \exists n \geq 1 \colon \rho(w)^n x = x \}$. In fact, since $\Omega_{x,uwu^{-1}} =\Omega_{\rho(u)^{-1}x, w}$, it suffices to take the intersection over all $w \in \Gamma_0$ which are cyclically reduced. 

Thus let $w = a_1 \dots a_l$ with $a_i \in \{a, a^{-1}, b_1, b_1^{-1},\dots \}$ be cyclically reduced and $x \in \bZ$.
It is again easy to see that each $\Omega_{x,w}$ is open. To show that each $\Omega_{x,w}$ is dense, we fix $\rho \in \Hom_\sigma(\Gamma,G)$  and $F \subset_f \Z$. We then have to construct $\eta \in \Hom_\sigma(\Gamma,G)$ with $\eta(b_i)|_F = \rho(b_i)|_F$ for all $i\geq 1$ such that there is $n \geq 1$ with $\eta(w)^n x = x$.

Roughly speaking, the construction of $\eta$ can be done as follows: If the $\rho(w)$-orbit of $x$ is not finite, then the sequence $x, \rho(a_l)x,\rho(a_{l-1} a_l) x,\dots$ contains infinitely many different elements, so will eventually reach a region far away from $F$. There we have all the freedom to set the $\eta(b_i)$ to ensure that after the word is finished, the $\rho(w)$-orbit of $x$ became periodic. We leave the details to the reader.

\item (\ref{itm:atomic tnf}) It suffices to show that for a generic representation the stabilizer map is injective. We can write 
 \begin{align*}
    \{\rho \in \Hom_\sigma(\Gamma,G) \mid \Stab \colon \Z \to \Sub(\Gamma) \text{ injective} \} &= \bigcap_{x,y \in \Z, x \neq y} \Omega_{x,y},
\end{align*}
where $\Omega_{x,y} \coloneqq  \{\rho \mid \Stab_{\Gamma}(x) \neq \Stab_{\Gamma}(y)\}$. It is again easy to see that each $\Omega_{x,y}$ is open, and it remains to show that each $\Omega_{x,y}$ is dense.

Thus fix $x,y \in \bZ$ and let $\rho \in \Hom_\sigma(\Gamma,G)$ and $F \subset_f \Z$. We have to find $\eta \in U_F(\rho)$ and $w \in \Gamma$ with $\eta(w)x=x$ and $\eta(w)y \neq y$. We will actually choose $\eta$ with $\eta(b_i) = \rho(b_i)$ for all $i\geq 2$ and only construct $\eta(b_1)$ carefully. For this we pick $z \in \bZ \setminus\{x,y\}$ and choose $n >0$ such that $\sigma^n(\{x,y,z\}) \cap (F \cup \rho(b_1)(F)) = \emptyset$. 
We now define an injection $\tau_0\colon F \cup \{\sigma^n(x), \sigma^n(y)\} \hookrightarrow \bZ$ such that $\tau_0|_F = \rho(b_1)|_F$ and $\tau_0(\sigma^n(x)) = x$ and $\tau_0(\sigma^n(y)) = \sigma^n(z)$ and extend it to an element $\tau$ of $G$. If we now set $\eta(b_1) \coloneqq \tau$ and $w\coloneqq a^{-n} b_1 a^n$, then $\eta(w) = \sigma^{-n}\tau \sigma^n$ and hence $\eta(w)x = x$ and $\eta(w) y \neq y$.

\item (\ref{itm:atomic boomerang}) 
We write
        \begin{align*}
           & \{ \rho \in \Hom_\sigma(\Gamma,G) \mid \forall x \in \Z\colon \Stab(x) \in \Boom(\Gamma) \} \\
            =& \{\rho \in \Hom_\sigma(\Gamma,G) \mid \forall x \in \Z \forall \gamma \in \Gamma \forall F \subset_f \Gamma \exists n \geq 1 \colon \Stab(x) \cap F = \gamma^n \Stab(x) \gamma^{-n} \cap F  \}\\
            =& \bigcap_{x \in \Z} \bigcap_{\gamma \in \Gamma} \bigcap_{F \subset_f \Gamma} \{\rho \in \Hom_\sigma(\Gamma,G) \mid \exists n \geq 1 \colon \Stab(x) \cap F = \Stab(\gamma^n x) \cap F \},
        \end{align*}
        and denote $\Omega_{x,\gamma,F} \coloneqq  \{\rho \in \Hom_\sigma(\Gamma,G) \mid \exists n \geq 1: \Stab(x) \cap F = \Stab(\gamma^n x) \cap F \}$. We have to show that $\Omega_{x,\gamma,F}$ is open and dense.

        To show that it is open, note that for every $\rho \in \Omega_{x,\gamma,F}$, every $\rho' \in \Hom_\sigma(\Gamma,G)$ with $\rho'|_F = \rho|_F$ is also contained in $\Omega_{x,\gamma,F}$; and these $\rho'$ are an open neighborhood of $\rho$.
        
        For a generic $\rho$, we know already that $\Omega_{x,\gamma,F}$ is dense whenever $\gamma$ is not conjugate to a power of $a$, since $\langle \gamma \rangle x$ is finite by (\ref{itm:atomic finorb}).
        Then, after conjugating and using that $\Omega_{x,\delta \gamma \delta^{-1},F} = \Omega_{\delta^{-1}x,\gamma,\delta^{-1} F \delta}$, we can assume that $\gamma = a^k$ for some $k \in \Z$.

        Let $\rho \in \Hom_\sigma(\Gamma,G)$.
        By Lemma \ref{WarmupLemma} 
        we can choose $\eta$ close to $\rho$ such that every element of $\eta(\Gamma)$ commutes with $\sigma^{n}$ for a very large $n$.
        But then
        \begin{align*}
            \Stab_{\eta(\Gamma)}(x) \cap F &=
            \{\gamma' \in F \mid \eta(\gamma')(x) = x\} \\
            &=  \{\gamma' \in F \mid  \eta(\gamma')\circ \sigma^{nk}(x) = \sigma^{nk}(x)\} \\
            &= \{\gamma' \in F \mid \eta(\gamma')(x+nk) = x+nk\} \\
            &= \Stab_{\eta(\Gamma)}(\eta(a^k)^n x) \cap F,
        \end{align*}
        which is what we needed to achieve.
\item (\ref{itm: atomic afht}) \emph{Faithful:} It is easy to check that for every $\gamma \in \Gamma$ the subset 
\[
\Omega_{\gamma} = \{\rho \in \Hom_\sigma(\Gamma, G) \mid \exists x \in \Z \colon \rho(\gamma)(x) \neq x \}
\]
is open and dense, and the intersection of these sets is the space of injective restricted representations.

\noindent \emph{Highly transitive:} By definition of the topology, a subgroup of 
$G = \mathrm{Sym}(\bZ)$ is highly transitive if and only if it is dense. Now
the set of constrained representations with dense image is 
\[
\Hom^{\mathrm{dense}}_\sigma(\Gamma,G) = \bigcap_{F \subset_f \Z} \bigcap_{\iota \colon F \hookrightarrow \Z} \Omega_{F,\iota}, \; \text{where} \; \Omega_{F,\iota} \coloneqq   \{\rho \in \Hom_\sigma(\Gamma,G) \mid \exists w \in \Gamma \colon \rho(w)|_F = \iota \}.
\]
Since each $\Omega_{F,\iota}$ is open, it is thus $G_\delta$, and it remains to show that $\Omega_{F,\iota}$ is dense for a fixed $F \subset_f \mathbb{Z}$ and $\iota \colon F \hookrightarrow \Z$. For this we fix $F$ and $\iota$; let $F' \subset_f \Z$ and $\rho \in \Hom_\sigma(\Gamma,G)$ and choose $n > 0$ such that
\[
\sigma^n(F) \cap F' = \emptyset \qand \sigma^n(\iota(F)) \cap \rho(b_1)(F') = \emptyset.
\]
Then $\sigma^n \circ \iota \circ \sigma^{-n}|_{\sigma^n(F)}: \sigma^n(F) \hookrightarrow \Z$ extends to an element $g$ of $G$, and there is a unique $\eta \in \Hom_\sigma(\Gamma,G)$ such that $\eta(b_1) = g$ and $\eta(b_j) = \rho(b_j)$ for $j \geq 2$. Moreover,
\[
\eta(a^{-n} b_1 a^n)|_F = \iota \implies \eta \in \Omega_{F,\iota} \cap U_{F'}(\rho) \implies \Omega_{F,\iota} \text{ is dense}.
\]

\noindent \emph{Amenable:}
The $G_\delta$ part follows from Corollary \ref{prop:Gdelta properties in the Chabauty}.
Recall that $\Gamma$ acts amenably if every finitely generated subgroup acts amenably, and $\bigcup_{r \geq 1} \Hom_\sigma(\Gamma_r,G)$ is dense in $\Hom_\sigma(\Gamma,G)$, so we assume without loss of generality that $\Gamma = \Gamma_r$ for some $r \geq 1$.
Let $\rho \in \Hom_\sigma(\Gamma,G)$.
By Lemma \ref{WarmupLemma}(1), there exist $h_1,\dots,h_{r-1} \in G$ arbitrarily close to $\rho(b_1),\dots,\rho(b_{r-1})$ and $k' > 0$ such that
each $h_i$ leaves each of the intervals $((2n-1)k',(2n+1)k']$, with $n \in \Z$, invariant.
Define $\eta \in \Hom_\sigma(\Gamma,G)$ via $\eta(b_i) := h_i$.
Then, the intervals $A_m = \bigcup_{n=-m}^m ((2n-1)k',(2n+1)k']$ form a F{\o}lner sequence for the action of $T$ and all $h_i$, hence the action of $\Gamma_r$ via $\eta$ and we are done.
\end{proof}

\begin{remark*}\label{IInftyNoGood}
Note that the argument in the proof of Part \eqref{itm:atomic finorb} makes essential use of the fact that singletons are wandering sets of positive measure for $\sigma$, so it appears to not be generalizable to the non-atomic case.
\end{remark*}
\end{no}
\section{Generic restricted representations are e.c.}\label{Sece.c.}

\subsection{Statement of the main result}
\begin{no}
We now work towards the proof of Theorem \ref{GCR1}. We will first establish the most difficult part of the theorem, which is the e.c.\ property. The main difficulty to overcome is that in view of 
Remark \ref{IInftyNoGood} we cannot argue as in the $\I_\infty$-case. Once the e.c. property is established almost all other parts follow quite easily. The one exception is density of generic representations which we will establish in Section \ref{SecDensity} below under addditional restrictions on the Krieger type. 

Throughout this section, $\cR$ denotes a hyperfinite and ergodic countable Borel equivalence relation with measurable full group $G$ on a non-atomic standard probability space $(X, \mu)$, and $T$ is some fixed choice of generator for $\cR$. Moreover, $\Gamma$ denotes a countable non-abelian free group, and we consider the Polish space  $\mathrm{Hom}_T(\Gamma, G)$ of constrained representations of $\Gamma$ in $G$.
\end{no}

\begin{no}\label{ECMainThm} We denote by $\mathrm{Hom}^{\mathrm{ec}}_T(\Gamma, G) \subset \mathrm{Hom}_T(\Gamma, G)$ the subspace of 
elementwise conservative restricted representations as given by
\[
 \mathrm{Hom}^{\mathrm{ec}}_T(\Gamma, G)= \bigcap_{w \in \Gamma} \Omega_w, \; \text{where}\; \Omega_w \coloneqq  \{\rho \in \mathrm{Hom}_T(\Gamma, G) \mid \rho(w) \text{ is conservative}\}.
\]
Then the first main result of this section can be stated as follows:
\begin{theorem*}
    The subset $\mathrm{Hom}^{\mathrm{ec}}_T(\Gamma, G)\subset \mathrm{Hom}_T(\Gamma, G)$ is a dense $G_\delta$-subset.
\end{theorem*}
\end{no}
\begin{no} By the Baire category theorem it suffices so show that for every $w \in \Gamma$ the subset $\Omega_w \subset  \mathrm{Hom}_T(\Gamma, G)$ is dense $G_\delta$. Establishing the $G_\delta$-property is actually elementary and can be reduced to the following observation:
\begin{lemma*}\label{ConOpen} The subset $G^{\mathrm{con}} \coloneqq  \{g \in G \mid g \text{ conservative}\}\subset G$ is a $G_\delta$-set.
\end{lemma*}
\begin{proof} Recall from Proposition \ref{ConservativeAll} that $g \in G$ is conservative if and only if for every measurable set $A \subset X$ we have $\mu(A \setminus \bigcup_{n \geq 1} g^{-n}(A)) = 0$, hence
\[
G^{\mathrm{con}} = \bigcap_{A \in \cB} \bigcap_{N \in \bN} \bigcap_{k \geq 1} \Omega_{A,N,k}, \quad \text{where} \quad \Omega_{A,N,k} \coloneqq  \Big\{g \in G \mid \mu(A \setminus \bigcup_{n = 1}^N g^{-n}(A)) < \frac{1}{k}\Big\}. 
\]
It is clear that $\Omega_{A,N,k} \subset G$ is open; the problem is that $\cB$ is not countable. To remedy this we choose a countable set $\cU$ of measurable sets that represent a dense subset of the measure algebra. We claim that for every $N \geq 1$ we have
\[
\bigcap_{A \in \cB} \bigcap_k \Omega_{A,N,k}= \bigcap_k \bigcap_{A \in \cU} \Omega_{A,N,k}.
\]
Indeed, the inclusion $\subseteq$ is clear and to show $\supseteq$, it suffices to prove that for every $k\geq 1$ we have $ \bigcap_{A \in \cB} \Omega_{A,N,k} \supset \bigcap_{A \in \cU} \Omega_{A,N,k+1}$. This, however, is true since the action of $g$ on the measure algebra is continuous: If there exists $A \in \cB$ with $\mu(A \setminus \bigcup_{n = 1}^N g^{-n}(A)) \geq \frac{1}{k}$ then for every $\epsilon > 0$ there exists $U \in \cU$ with $\mu(U \setminus \bigcup_{n = 1}^N g^{-n}(U)) \geq \frac{1}{k} - \epsilon$.
\end{proof}
\begin{corollary*}\label{ConOpen2} The subset $\Omega_w \subset \mathrm{Hom}_T(\Gamma, G)$ is $G_\delta$ for every $w \in \Gamma$, and hence also the subset $ \mathrm{Hom}^{\mathrm{ec}}_T(\Gamma, G)\subset  \mathrm{Hom}_T(\Gamma, G)$ is a $G_\delta$-set.
\end{corollary*}
\begin{proof} The evaluation map $\mathrm{ev}_w \colon \mathrm{Hom}^{\mathrm{ec}}_T(\Gamma, G) \to G$ is continuous (see  \S \ref{Ev}), and 
$\Omega_w = \mathrm{ev}_w^{-1}(G^{\mathrm{con}})$, hence the corollary follows from Lemma \ref{ConOpen}.
\end{proof}
In view of the Baire category theorem 
we have thus reduced the proof of Theorem \ref{ECMainThm} to showing that each of the subsets $\Omega_w \subset \mathrm{Hom}_T(\Gamma, G)$ is dense. We will establish this density in Subsection \ref{SecZebras} below; the proof is based on a purely combinatorial statement about permutations which we establish in the next subsection.
\end{no}

\subsection{A combinatorial lemma}\label{SecComLemma}
\begin{no}\label{sec:interval} 
Recall our notation for intervals in the integers: Given integers $a \leq b$, we set
\[
[a,b] := \{a, a+1, \dots, b-1, b\}.
\]
If $c,d,e,f$ are integers with $c \leq d = e-1<f$, then we say that $[c,d]$ is \emph{left-adjacent} to $[e,f]$ (or that $[e,f]$ is \emph{right-adjacent} to $[c,d]$). We will be interested in a special class of permutations of $\bZ$ defined as follows:
\begin{definition*}
A collection $\cZ = ((B_n)_{n \in \Z}, (W_n)_{n \in \Z})$ of intervals is called a \emph{zebra strip decomposition}
if \[
\Z = \dots \sqcup W_{k-1} \sqcup B_k \sqcup W_{k} \sqcup B_{k+1} \sqcup \dots,
\]
and each $W_k$ is right-adjacent to $B_{k}$ and left-adjacent to $B_{k+1}$. In this case the sets $W_k$ (respectively $B_k$) are called the \emph{white stripes} (respectively \emph{black stripes}) of $\cZ$ and the
integer
\[
s(\cZ) \coloneqq  \min\{|B_k|\mid k \in \Z\}.
\]
is called its \emph{minimal strip width} of $\cZ$. A permutation $g$ of $\Z$ is called a \emph{zebra permutation} of type $\cZ$ if it preserves each of the sets $W_k$ and fixes each of the sets $B_k$ pointwise.
\end{definition*}
\end{no}

\begin{figure}[h!]
\begin{tikzpicture}[x=0.7cm, y=0.7cm, every node/.style={font=\small}]

\def\xmin{1}
\def\xmax{20}
\draw[thick] (\xmin,0) -- (\xmax,0);

\foreach \i in {1,...,20}{
  \coordinate (p\i) at (\i,0);
  \filldraw[black] (p\i) circle (0.07cm);
}

%

\foreach \xlo/\xhi/\label in {
  0.5/3.5/{$B_{-1}$},
  7.5/10.5/{$B_{0}$},
  14.5/17.5/{$B_{1}$},
  19.5/20.5/{$B_{2}$}
}{
  \filldraw[black, rounded corners=1pt] (\xlo,-0.6) rectangle (\xhi,-0.3);
  \node[below=9pt] at ({(\xlo+\xhi)/2},-0.6) {\label};
}

\node[below=9pt] at (5.5,-0.6)  {$W_{-1}$};
\node[below=9pt] at (12.8,-0.6) {$W_{0}$};
\node[below=9pt] at (18,-0.6) {$W_{1}$};


\newcommand{\permRightArrowAbove}[2]{%
  \coordinate (src) at (p#1);
  \coordinate (tgt) at (p#2);
  \draw[gray!70!black, ->, line width=0.7pt, shorten >=1pt]
    (src) to[out=65, in=115, looseness=1.5] (tgt);
}

\newcommand{\permLeftArrowAbove}[2]{%
  \coordinate (src) at (p#1);
  \coordinate (tgt) at (p#2);
  \draw[gray!70!black, ->, line width=0.7pt, shorten >=1pt]
    (src) to[out=115, in=65, looseness=1.5] (tgt);
}

\newcommand{\permRightArrowBelow}[2]{%
  \coordinate (src) at (p#1);
  \coordinate (tgt) at (p#2);
  \draw[gray!70!black, ->, line width=0.7pt, shorten >=1pt]
    (src) to[out=-115, in=-65, looseness=1.5] (tgt);
}

\newcommand{\permLeftArrowBelow}[2]{%
  \coordinate (src) at (p#1);
  \coordinate (tgt) at (p#2);
  \draw[gray!70!black, ->, line width=0.7pt, shorten >=1pt]
    (src) to[out=-115, in=-65, looseness=1.5] (tgt);
}

\newcommand{\permDoubleArrowAbove}[2]{%
  \coordinate (src) at (p#1);
  \coordinate (tgt) at (p#2);
  \draw[gray!70!black, <->, line width=0.7pt, shorten >=1pt]
    (src) to[out=65, in=115, looseness=1.5] (tgt);
}

\permRightArrowAbove{4}{5}
\permRightArrowAbove{5}{7}
\permLeftArrowBelow{6}{4}
\permLeftArrowBelow{7}{6}

\permDoubleArrowAbove{11}{14}
\permDoubleArrowAbove{12}{13}

\permDoubleArrowAbove{18}{19}
\end{tikzpicture}
\caption{\label{fig:zebra} A zebra permutation}
\end{figure}
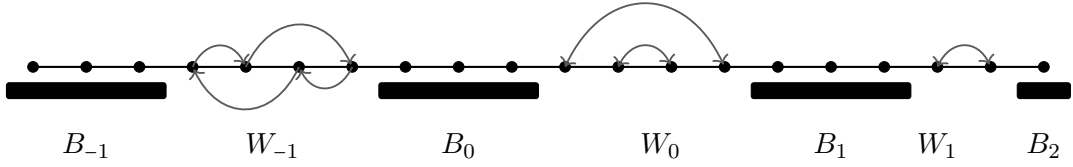

\begin{no} We now formulate our main combinatorial lemma concerning zebra permutations. Let $c_{a} \colon \Gamma \to \Z$ denote the $a$-counting homomorphism, i.e.\ the unique homomorphism with  $c_a(a)=1$ and $c_a(b_i)=0$ for all $i$. If $g$ is a permutation of $\Z$, then we define its \emph{action radius} by
\begin{equation}\label{ActionRadius}  R(g) \coloneqq  \sup_{n \in \Z}|g(n) - n| \in [0, \infty] \quad \text{so that}\quad  R(g_1g_2) \leq R(g_1) + R(g_2).\end{equation}

\begin{lemma*} \label{Combinatoric} Consider a $\Gamma_r$-action on $\Z$ such that $a$ acts by $n \mapsto n+1$ and $b_1, \dots, b_{r-1}$ act by zebra permutations of the same type $\cZ$. If $w \in \Gamma_r$ satisfies
\begin{equation}\label{WAssumption}
\|w\| < \frac{s(\cZ)}{3R}, \quad \text{where}\quad R \coloneqq   \sup\{R(b_j) \mid j \in\{1, \dots, r-1\}\},
\end{equation}
then $w$ has exactly $|c_{a}(w)|$ infinite orbits.
\end{lemma*}
\begin{remark*}\label{balanced} If $w$ is \emph{$a$-balanced} in the sense that $c_{a}(w)=0$, then the conclusion of Lemma \ref{Combinatoric} is that $w$ acts with finite orbits.
 \end{remark*}
\end{no}

\begin{no}
 The remainder of this subsection is devoted to the proof of Lemma \ref{Combinatoric}. We will assume that $c_a(w) \geq 0$ which can always be achieved by replacing $w$ by $w^{-1}$ if necessary. We abbreviate $L \coloneqq  \|w\|$ and write
\[
w = s_1 \cdots s_L \text{ with }s_i \in \{a^{\pm 1}, b_1^{\pm 1}, \dots b_{r-1}^{\pm 1}\}.
\]
Since $R(a) = 1$, we deduce from \eqref{ActionRadius} that $R(w) \leq L R$, hence \eqref{WAssumption} yields
\begin{equation}\label{sigmawupper1}
c_a(w) \leq R(w) \leq LR \leq \lfloor s(\cZ)/3 \rfloor
\end{equation}
We decompose each $B_k$ as $B_k = L_k \sqcup C_k \sqcup R_k$, where $L_k$ and $R_k$ denote the first, respectively last $LR$ points of $B_k$. Then
\[\bZ = \dots \sqcup R_{k-1} \sqcup W_{k-1} \sqcup L_k \sqcup C_k \sqcup R_k \sqcup W_{k} \sqcup L_{k+1} \sqcup \dots
\] 
 and by \eqref{sigmawupper1} we have
\begin{equation}\label{StripsAreWide}
R(w) \leq LR \leq \min\{|L_k|, |C_k|, |R_k|\} \text{ for all }k \in \bZ.
\end{equation}
For every $k \in \Z$ we define $R_k^+ \coloneqq  R_k \sqcup W_k \sqcup L_{k+1} \sqcup C_{k+1} \sqcup R_{k+1} \sqcup\dots$.
\begin{lemma*}\label{ZebraDyn}
With notations as above the following hold for every $k \in \Z$:
\begin{enumerate}[(i)]
\item $wW_k \subset R_k \sqcup W_{k} \sqcup L_{k+1}$.
\item $wL_k \subset R_{k-1} \sqcup W_{k-1} \sqcup L_k \sqcup C_k$.
\item $wC_k \subset C_k \sqcup R_k$; more precisely, if $n \in C_k$, then $wn = n + c_{a}(w) > n$.
\item $wR_k \subset R_k \sqcup W_k \sqcup L_{k+1}$.
\item $wR_k^+ \subset R_k^+$.
\end{enumerate}
\end{lemma*}
\begin{proof} By \eqref{StripsAreWide} no $w$-orbit can jump over one of the strips $L_k$, $C_k$ or $R_k$ without entering it. This implies (i) and (ii), and (v) is immediate from (i)-(iv). For the proof of (iii) (respectively (iv)) we pick $n \in C_k$ (respectively $n \in R_k$) and define
\[
n_0 \coloneqq  n, \; n_1\coloneqq  s_Ln_0, \; n_2 \coloneqq  s_{L-1}s_Ln_0, \; \dots, \; n_L \coloneqq  s_{1} \cdots s_{L}n_0 = wn.
\]
\item (iii) Since $b_1, \dots, b_{r-1}$ fix $B_k$ pointwise, a straight-forward induction shows that
\begin{equation}\label{InductionZebra}
n_j \in B_k\qand n_0 \leq n_j \leq n_0+j \quad \text{ for all } j = 0, \dots, L.
\end{equation}
 Since $a$ acts by the shift and $b_1, \dots, b_{r-1}$ fix $B_k$ pointwise it follows that $wn = n+c_a(w)$.\\

\item (iv) By \eqref{StripsAreWide} we have $n_j \in C_k \sqcup R_k \sqcup W_k \sqcup L_{k+1}$ for all $j \in \{1, \dots, L\}$. If $n_L  = wn\in W_k 
\sqcup L_{k+1}$, then we are done, thus assume that $n_L \in C_k \sqcup R_k$. We distinguish two cases.

\item If $\{n_0, \dots, n_L\} \subset C_k \sqcup R_k \subset B_k$, then $wn = n_L \in C_k \sqcup R_k$ and
$b_1, \dots, b_{r-1}$ act trivially on each $n_j$, and hence $wn = n + c_a(w) \geq n$, which implies $wn \in R_k$. Otherwise there exists $j$ such that $n_j \in W_k \sqcup L_{k+1}$ and $\{n_{j+1}, \dots, n_{L}\} \in C_k \sqcup R_k \subset B_k$. We then obtain $n_{j+1} \geq n_j - \sigma$ and since every letter maps an element of $B_k$ by at most $1$ we get $n_{j+1+m} \geq n_j - \sigma - m$ for $m \leq L-j-1$. In particular, $wn \geq n_j - L \sigma$, and since $n_j \in W_k\sqcup L_{k+1}$ and $R_k$ has width $L\sigma$ we cannot have $wn \in C_k$; thus $wn \in R_k$.
\end{proof}
\begin{corollary*}\label{FwdOrbitAll} Let $k_o \in \Z$ and $m \coloneqq  \min C_{k_o}$. Then the $w$-orbits of $m$, $m+1$, \dots, $m+c_a(w) - 1$ are pairwise disjoint, cover $C_{k_o}$ and intersect $C_k$ for every $k \in \Z$.
\end{corollary*}
\begin{proof} Set $I \coloneqq  \{m, \dots, m+c_a(w)\}$, let $n \in I$ and denote
\[
\cO_n^{\pm} \coloneqq  \{w^kn \mid \pm k \geq 0\} \qand \cO_n \coloneqq  \{w^kn \mid k \in \Z\}.
\]
By Lemma \ref{ZebraDyn}(iii) there exists $l\in \bN$ such that $wn, \dots, w^{l-1}n \in C_{k_o}$ and $w^{l}n \in R_{k_o}$. It then follows from Lemma \ref{ZebraDyn}(v) that all other points of $\cO_n^+$ land in $R_{k_o}^+$. We deduce that
\begin{equation}\label{FwdOrbits}
\cO_n^+ \cap I = \{n\}, \quad \cO_n^+ \cap C_{k_o} = \{n, n+c_a(w), \dots, n + (l-1)c_a(w)\} \qand n = \min \cO_n^+
g  \end{equation}
The latter forces $\cO^+_n$ to be infinite; since it is bounded below, it must be unbounded from above, hence intersect $C_k$ for every $k \geq k_0$ by \eqref{StripsAreWide}. By a symmetric argument (involving $w^{-1}$ instead of $w$), $\cO$ also intersects $C_k$ for every $k<k_o$. 

It also follow from \eqref{FwdOrbits} that for $n' \neq n \in I$ we have $n' \not \in \cO^+_n$ and $n \not \in \cO^+_{n'}$, hence the orbits of $n$ and $n'$ are different, hence disjoint. 
\end{proof}
\begin{proof}[Proof of Lemma \ref{Combinatoric}] By \eqref{StripsAreWide}, every infinite orbit must intersect some $C_k$, hence by Corollary \ref{FwdOrbitAll} it must intersect $C_0$. By the same corollary, there are precisely $c_a(w)$ orbits which intersect $C_0$ and all of them are infinite, hence the lemma follows.
\end{proof}
\end{no}

\subsection{Zebras in the full group}\label{SecZebras}
\begin{no}\label{FTD}
    We now return to the setting of Theorem \ref{GCR1}; in particular $G$ denotes the measurable full group of $\cR$ and $T$ denotes a fixed choice of generator. We denote by $\alpha\colon G \times X \to \mathrm{Sym}(\bZ)$ the action cocycle with respect to $T$ as in \S \ref{ActionCocycle}. We say that $g \in G$ has \emph{total displacement at most $m$} if for almost every $x \in X$ and for every $k \in \Z$ the permutation $\alpha_x(g^k)$ has action radius at most $m$. This means that for almost all $x \in X$ we have 
\[\langle g \rangle x \subset \{T^{-m}x, \dots, T^mx\}.\] If such an $m$ exists we say that $g$ has \emph{finite total displacement}; note that this implies, in particular, that $g$ is periodic, and hence $g$ is conservative.
 \begin{proposition*}
        The set of elements with finite total displacement is dense in $G$.
    \end{proposition*}

    \begin{proof}
        Let $g \in G$ and $\epsilon > 0$. We need to construct $h \in G$ with finite total displacement and $d(g,h) < \epsilon$. We will assume that $g$ is periodic; since elements with this property are dense this is no loss of generality.
        
        Define $X_m := \{x \in X \mid \langle g \rangle x \subset \{T^{-m}x,\dots,T^{m}x\}\}$. Since almost all orbits of $g$ are finite, we have that
        $X = \bigcup_{m \geq 0} X_m$, and this is an ascending union of $g$-invariant sets.
        We may thus choose $m > 0$ such that $\mu(X_m) > 1 - \epsilon$.
        Then
        \[h(x) = 
        \begin{cases}
            g(x) & x \in X_m\\
            x & x \notin X_m
        \end{cases}
        \]
        has total displacement at most $m$ and satisfies $d(g,h) < \epsilon$.
    \end{proof}
Note that the set of elements with finite total displacement is not $G_\delta$.
\end{no}
\begin{no}\label{ZebrasDense} 
For every $x \in X$ we have a canonical bijection $\iota_x \colon \Z \to [x]_{\cR}$, $n \mapsto T^n x$. Every $B \subset X$ thus defines a subset $B_x := \iota_x^{-1}(B \cap [x]_{\cR}) \subset \bZ$, and the intervals in $B_x$ form the black stripes of a zebra strip decomposition $\cZ_x(B)$ of $\bZ$.

We now say that $(g_1, \dots, g_n) \in G^n$ is a \emph{dazzle of zebras} of \emph{uniform strip width} $\geq M$  if
there exists a subset $B \subset X$ of positive measure such that for almost every $x \in X$ the associated zebra strip decomposition $\cZ_x = \cZ_x(B)$ satisfies $s(\cZ_x) \geq M$ and $\alpha_x(g_1), \dots, \alpha_x(g_n)$ are zebra permutations of type $\cZ_x$. We can then upgrade Proposition \ref{FTD} as follows:
\begin{proposition*}
For every $n \in \bN$ and $\epsilon > 0$ there exists $m \in \bN$ such that for all $M > m$ and all $(h_1, \dots, h_n) \in G^n$ there exists a dazzle of zebras $(f_1, \dots, f_n)$ of total displacement $\leq m$ and uniform strip width $\geq M$ such that $d(f_j, h_j) < \epsilon$ for all $j \in \{1, \dots, n\}$.
\end{proposition*}
For the proof of the proposition we need some general properties of first return maps. If $g \in G$ is conservative, then for every positive measure subset $A \subset X$ the first-return map to $A$ is well-defined, and we an extend it to a non-singular transformation of $X$ by setting
\[
g_A(x) := \begin{cases}
    g^k(x), & \text{if } x \in A \text{ and } k = \min\{k \geq 1 \mid g^k(x) \in A\}, \\
    x, & \text{else.}
\end{cases}
\]
Note that if $g$ has total displacement at most $M$, then so does $g_A$. Moreover, since $g_A(x) = g(x)$ for all $x \in A \cap g^{-1}(A)$ we have
$d(g, g_A) \leq 1 - \mu(A \cap g^{-1}(A))$. In particular, if $\mu(B) < 1$, then
then applying this to $A := X \setminus B$ shows that
\begin{equation}\label{ReturnDistance}
   d(g, g_{X \setminus B}) \leq \mu(B \cup g^{-1}B). 
\end{equation}
\begin{proof}[Proof of Proposition \ref{ZebrasDense}] Fix $\eps > 0$ and 
$(h_1, \dots, h_n) \in G^n$. By Proposition \ref{FTD} we may assume that $h_1, \dots, h_n$ all have total displacement $\leq m$. We may also assume that $\eps < 1$.

By Lemma \ref{par:nonsingular on ma} we find $\delta > 0$ with the following property: If $A \subset X$ is Borel with $0< \mu(A) < \delta$, then for all $h \in \{e, h_1, \dots, h_n\}$ the sets $h^{-1}A, h^{-1}TA, \dots, h^{-1}T^M A$ have measure at most $\eps/2(M+1)$. Now fix $A$ with $0 < \mu(A) < \delta$ and define $B := A \cup \dots \cup T^M A$. Then we have 
\[
\mu(B \cup h_j^{-1}(B))\leq \sum_{i=0}^M \mu(T^iA) + \mu(h_j^{-1}T^iA)  \leq \eps \quad \text{for all $j \in \{1, \dots, n\}.$}
\]
Given $j \in \{1, \dots, n\}$ we now define $f_j \coloneqq (h_j)_{X \setminus B}$ and deduce from \eqref{ReturnDistance} that $d(h_j, f_j) \leq \eps$. Moreover, each $f_j$ has total displacement at most $m$. It remains to show that $(f_1, \dots, f_n)$ is a dazzle of zebras of uniform strip width $\geq M$.

Since $B = A \cup TA \cup \dots \cup T^M A$, for every $x \in A$ the associated zebra strip decompositions $\cZ_x = \cZ_x(B)$ contains a black interval containing $[0, M]$. Since $T$ is ergodic and $\mu$ is non-atomic, $T$ is convervative, and hence the $T$-orbit of $x$ either stays inside $B$ or reenters $A$ at some other time almost surely. This implies that $s(\cZ_x) \geq M$ for almost all $x \in X$.

Moreover, by construction $\alpha_x(f_1),\dots,\alpha_x(f_n)$ act trivially on the black stripes, and since these are of width $\geq M > m$ and each $\alpha_x(f_j)$ has action radius at most $m$, they must also preserve each white strip.
\end{proof}
\begin{remark*} The proof actually also yields the following statement: If $h_1, \dots, h_n$ have total displacement $\leq m$, we can choose $f_j \coloneqq (h_j)_{X \setminus B}$ where $B = A \cup \dots \cup T^M A$ for every Borel set $A$ of sufficiently small positive measure.
\end{remark*}
\end{no}
\begin{no}\label{ecproof} We can now prove Theorem \ref{ECMainThm}, which implies Part \eqref{itm:nonatomic ec} of Theorem \ref{GCR1}.
\begin{proof}[Proof of Theorem \ref{ECMainThm}]
In view of Corollary \ref{ConOpen2} it remains to show that for every $w \in \Gamma$ the subset 
\[
\Omega_w \coloneqq  \{\rho \in \mathrm{Hom}_T(\Gamma, G) \mid \rho(w) \text{ is conservative}\} \subset \mathrm{Hom}_T(\Gamma, G)
\]
is dense. Thus fix $w \in \Gamma$ and choose $n$ such that $w$ only contains the letters $a, b_1, \dots, b_n$. Given $j \in \{1, \dots, n\}$ we define $g_j := \rho(b_j)$. Now fix $\eps > 0$.

By Proposition \ref{ZebrasDense} we find $m \in \bN$ and a dazzle of zebras $(f_1, \dots, f_n) \in G^n$ of total displacement $\leq m$ and uniform strip width $M > 3 m \|w\|$ such that $d(f_j, g_j) \leq \eps$. Denote by $\rho' \in \Hom_T(F_r,G)$ the constrained representation with $\rho'(b_j)=f_j$ for $j \in \{1, \dots, n\}$ and $\rho'(b_j) = \rho(b_j)$ for $j>n$.

Now for almost every $x \in X$ the permutations $\alpha_x(f_1), \dots, \alpha_x(f_m)$ are zebra permutations of the same type $\cZ_x$. Moreover, we have
\[
R \coloneqq   \sup\{R(\alpha_x(f_j)) \mid j \in\{1, \dots, n\}\} \leq m
\implies \|w\| < \frac{M}{3m}\leq \frac{s(\cZ_x)}{3R}.\]
It thus follows from Lemma \ref{Combinatoric} that every $T$-orbit contains exactly $|c_a(w)|$ infinite $\rho'(w)$-orbits, hence $\rho'(w)$ is conservative by Proposition~\ref{Conservativity}. This shows that $\rho' \in \Omega_w$, and since $\rho'$ is $\epsilon$-close to $\rho$, the theorem follows.
\end{proof}
\begin{remark*} If $c_{a}(w)=0$, then in view of Remark \ref{balanced} this proof gives the stronger conclusion that 
\[ 
\Omega_w^{\mathrm{per}} \coloneqq  \{\rho \in \mathrm{Hom}_T(\Gamma, G) \mid \rho(w) \text{ is periodic}\}\]
is dense. This also holds for many other words of a special form. Denote by $D \subset G$ the dense $G_\delta$ subset of periodic elements. For simplicity, we only present the case $r = 2$ and denote $b_1 = b$.
\begin{enumerate}
    \item \label{itm:primitive} If $w$ is part of a free basis of $\Gamma$ together with $a$, then the evaluation map
    $w(T,\cdot) \colon G \to G$ is a homeomorphism.
    Therefore, $(w(T,\cdot))^{-1}(D)$ is dense $G_\delta$. 
    \item \label{itm:balanced} For words of the form $w = a^m b^n a^{m'}$ with $m,n,m' \in \bZ$ and $n \neq 0$, we have $w(T,\cdot) = w'(T,\cdot) \circ p^{n}$, where $p^n(g) = g^n$ and with $w'$ as in the last point. We again have that $w(T,\cdot)^{-1}(D)$ is dense $G_\delta$ since $D = ((p^n)^T)^{-1}(D)$.
    \item Let now $w'$ be as in (\ref{itm:primitive}) and $w''$ as in (\ref{itm:balanced}). If we substitute every $b$ in $w''$ by $w'$, we get a word $w$ with $w(T,\cdot) = w''(T,\cdot) \circ w'(T,\cdot)$ and hence $w(T,\cdot)^{-1}(D)$ is dense. The words $w$ obtained satisfy that $c_a(w) = (m + m') \cdot c_b(w)$.
\end{enumerate}
For powers of $a$ and their conjugates, by definition, almost all orbits are infinite. We do not know whether these are the only words $w$ such that $w(T,\cdot)^{-1}(D)$ is not dense.
\end{remark*}
\end{no}
\subsection{Further properties of generic constrained representations}\label{SecFurther}
\begin{no}
The approximation argument used in the proof of Theorem \ref{ECMainThm} (and hence Part \eqref{itm:nonatomic ec} of Theorem \ref{GCR1}) can also be used to prove Parts \eqref{itm:nonatomic inj} and \eqref{itm:nonatomic afht} of Theorem \ref{GCR1}. For example, the first statement in \eqref{itm:nonatomic afht} can be proved as follows: 
\begin{proposition*}
    The set of $T$-constrained representations $\rho \colon \Gamma \to G$ such that $\rho(\Gamma)$ acts amenably (in the sense of Greenleaf) on almost every orbit is a dense $G_\delta$-set in $\mathrm{Hom}_T(\Gamma, G)$.
\end{proposition*}
\begin{proof} The $G_\delta$ property was already established in Corollary \ref{prop:Gdelta properties in the Chabauty}; for density we argue 
as follows. 

Let $\eps > 0$, $r \leq \mathrm{rk}(\Gamma)$ a positive integer and $\rho \in \mathrm{Hom}_T(\Gamma, G)$. We need to show that  
there exists a representation $\rho'$, which is $\eps$-close to $\rho$ and such that $\rho'(\Gamma_r)$ acts amenably. Given $j \in \{1, \dots, r-1\}$ we set $g_j := \rho(b_j)$, so that $\rho(\Gamma_r) = \langle T, g_2, \dots, g_r \rangle$. By Proposition \ref{ZebrasDense} we find $m \in \bN$ and a dazzle of zebras $(f_1, \dots, f_{r-1}) \in G^{r-1}$ of total displacement $\leq m$ and uniform strip width $M > m$ such that $d(f_j, g_j) \leq \eps$. 

Now let $\rho' \in \Hom_T(\Gamma, G)$ such that $\rho'(b_j)=f_j$ for $j \in \{1, \dots, r-1\}$ and $\rho'(b_j) = \rho(b_j)$ for $j\geq r$; then $\rho'$ is $\eps$-close to $\rho$.
The key observation is that Zebra permutations preserve arbitrarily large intervals (given by unions of consecutive white and black stripes). 
For almost every $x \in X$ the permutations $\alpha_x(f_j)$ are compatible zebra permutations, hence there exists sequences $n_k,m_k \to \infty$ such that the sets
\[
C_k(x) := \{T^j x \mid j \in [n_k, m_k]\} \subset [x]_{\cR}
\]
are invariant under all $\alpha_x(f_j)$. Since these sets are also F\o lner sets for the shift $\alpha_x(T)$ we deduce that the action of $\Gamma_r$ on $[x]_{\cR}$ under $\rho'$ is amenable.
\end{proof}
\end{no}
\begin{no} In order to establish also the rest of Parts \eqref{itm:nonatomic inj} and \eqref{itm:nonatomic afht} of Theorem \ref{GCR1}, we need to be able to insert permutations into a given zebra. For this we observe that if $A \subset X$ is a Borel subset and $I \subset \Z$ is interval such that the sets $T^iA$ with $i \in I$ are disjoint, then we can embed $\Sym(I)$ into $G$ via
\[
\iota_{A,I} \colon \Sym(I) \to G, \quad \iota_{A,I}(\pi)(x) = \begin{cases}
    T^{\pi(i)}(x) & \text{if }x \in T^i A \text{ for some } i \in I \\
    x & \text{else}.
\end{cases}
\]
If we now set $B := \bigsqcup_{i \in I} T^i A$, then for all $g \in G$ and $\pi \in \mathrm{Sym}(I)$ we have $g(x) =  \iota_{A, I}(\pi) \circ g_{X \setminus B} (x)$ for all $x \in B \cup g^{-1}B$. This implies that
\[
d(g, \iota_{A,I}(\pi) \circ g_{X \setminus B}) \leq \mu(B \cup g^{-1}(B)).
\]
Moreover, $\alpha(\iota_{A,I}(\pi) \circ g_{X \setminus B},x)|_I = \pi$.

Using this technique we can ``locally insert" permutations into constrained representations into $G$ in the following sense.

\begin{lemma*}\label{lem:instert words into constraineds}
    Let $\rho \in \Hom_T(\Gamma,G)$ and $I \subset \Z$.
    Let $r \geq 1$ and $w \in \Gamma_r \subset \Gamma$ be any word.
    Then for all $\pi_1,\dots,\pi_{r-1} \in \Sym(\Z)$, every subset $A \subset X$ with $\mu(A) > 0$ and every $\epsilon > 0$ there exist $\rho' \in \Hom_T(\Gamma,G)$ and $A' \subset A$ of positive measure such that
    \begin{enumerate}
        \item for all $i,j \in I$ with $i \neq j$ we have $\mu(T^iA' \cap T^jA') = 0$,
        \item for almost every $x \in A'$ we have $\alpha(\rho'(w),x)|_I = w(\sigma,\pi_1,\dots,\pi_{r-1})|_I$, and
        \item $d(\rho',\rho)<\epsilon$.
    \end{enumerate}
\end{lemma*}

\begin{proof}
    Set $\rho'(b_i) := \rho(b_i)$ for every $i \geq r$.

    Since the set of finitely supported permutations is dense in $\Sym(\Z)$, and by continuity of the evaluation map $w \colon \Sym(\Z)^r \to \Sym(\Z)$,
    we can, without changing the restriction of $w(\sigma,\pi_1,\dots,\pi_{r-1})$ to $I$, assume without loss of generality that $\pi_1,\dots,\pi_{r-1}$ have finite support.
    Note that it suffices to achieve (1) and (2) for some interval $I' \supset I$, so we can assume without loss of generality that $\supp(\pi_i) \subset I$ for every $1 \leq i < r$.
    Define $M_1 := \min(I) - \|w\|$ and $M_2 := \max(I) + \|w\|$, so that if we apply $w$ letter-by-letter to $\sigma,\pi_1,\dots,\pi_{r-1}$, the image of $I$ always remains contained in $[M_1,M_2]$.

    By Lemma \ref{par:nonsingular on ma} there exists $\epsilon' > 0$ such that for every $A'$ with $\mu(A') < \epsilon'$ and every $1 \leq i < r$ we have $\mu(A' \cup \rho(b_i)^{-1}(A')) < \epsilon$.
    By Lemma \ref{lem:many disjoint translates} we can choose $A' \subset A$ such that $T^{M_1}A',\dots,T^{M_2}A'$ are all disjoint and, setting $B := T^{M_1}A' \cup \dots \cup T^{M_2}A'$, we have $\mu(B) < \epsilon'$.
    We now finish defining $\rho'$ by setting $\rho'(b_i) := \iota_{A',[M_1,M_2]} \circ \rho(b_i)_{X \setminus B}$ for $1 \leq i < r$; this does the job.
\end{proof}

\end{no}

\begin{no}
We now recall from Lemma \ref{lem:faithful equals on every orbit} that a representation $\rho\colon \Gamma \to G$ is faithful if and only if it is faithful on every orbit. Thus the space of faithful restricted representations is given by
\[ \Hom^{\mathrm{ff}}_T(\Gamma,G) = 
\bigcap_{w \in \Gamma \setminus \{1\}} \Omega_w, \quad \text{where}\quad \Omega_w \coloneqq \{\rho \in \Hom_T(\Gamma,G) \mid \rho(w) \neq 1\}.
\]
We can now prove Part \eqref{itm:nonatomic inj} and the second statement of Part \eqref{itm:nonatomic afht} of Theorem \ref{GCR1}:
\begin{proposition*}
 The subset    $\Hom^{\mathrm{ff}}_T(\Gamma,G) \subset  \Hom_T(\Gamma,G)$ is a dense $G_\delta$-set.
\end{proposition*}
\begin{proof} 
Each of the sets $\Omega_w$ is open because the evaluation map $w \colon G^{r} \to G$ is continuous, hence $\Hom^{\mathrm{ff}}_T(\Gamma,G)$ is $G_\delta$. It remains to show that for every $w \in \Gamma$ the set $\Omega_w \subset \Hom_T(\Gamma,G)$ is dense.
So let $\rho \in \Hom_T(\Gamma,G)$ and $\epsilon > 0$.
Take $\tau \in \Hom_\sigma(\Gamma,\Sym(\Z))$ with $\tau(w)(0) \neq 0$.
By Lemma \ref{lem:instert words into constraineds}, applied with $\pi_i = \tau(b_i)$, we can find $\rho' \in \Hom_T(\Gamma,G)$ and a positive measure set $A$ with $\rho'(w)(A) = T^{\tau(w)(0)}A$ disjoint from $A$ and $d(\rho,\rho') < \epsilon$. This finishes the proof.
\end{proof}

\begin{remark*}
    Le Ma\^itre \cite[Thm. 6.2(1)]{LM:full_groups} showed that faithful constrained representations are dense and $G_\delta$ in the $\II_1$ case; we could also have generalized his argument. We would like to point out that in the more general case it is also stated in \cite[Proposition 4.7]{Mercer1993}; but there are problems in the proof that should not be difficult to fix.
\end{remark*}   

\begin{remark*}
    High transitivity on almost every orbit could be proven similarly.
    We get it from \S\ref{par:high trans from dense} in types $\II_\infty$ and $\III_\lambda$, the type $\II_1$ case is covered in the proof in \cite[Section 6]{EG:generic_irs}.
\end{remark*}
\end{no}

\begin{no}

It was shown by Bowen \cite[Theorem  5.3]{Bowen:free_irs}, in a more general setting, that the set of totally-non-free representations is a $G_\delta$ set in $\Hom(\Gamma,G)$. 
In fact, Bowen shows that the set of totally-non-free representations $\rho$ such that $\rho(a)$ is aperiodic is dense $G_\delta$ in the set of all representations such that $\rho(a)$ is aperiodic. We prove something very similar here.

\begin{proposition*}\label{prop-tnf}
    The set $\{\rho \in \Hom_T(\Gamma,G) \mid \Gamma \text{ acts totally-non-freely via } \rho \}$ is dense $G_\delta$ in $\Hom_T(\Gamma,G)$.
\end{proposition*}

\begin{proof}
    Let $\{A_n \mid n \geq 0\}$ be a countable dense set in the measure algebra.
    A representation $\rho$ induces a totally-non-free action if and only if it is not constant on any $A_n$. We can write the set of such representations as countable intersection
    \begin{align*}
        \Hom_T^{\mathrm{tnf}}(\Gamma,G) &= \{\rho \in \Hom_T(\Gamma,G) \mid \forall n \geq 0 \colon \, \Stab|_{A_n} \text{ not constant} \} \\
        &= \bigcap_{n \geq 0} \bigcup_{w \in \Gamma} \{\rho \in \Hom_T(\Gamma,G) \mid 0 < \mu(A_n \cap \Fix(\rho(w)) < \mu(A_n)\},
    \end{align*}
    where $\Fix(g)$ denotes the set of fixed points of $g \in G$.
    It is not difficult to see that these sets are open. We want to show that they are also dense. Fix $n$. Take $\rho \in \Hom_T(\Gamma,G)$ and $\epsilon > 0$ with $\epsilon < \mu(A_n)$.
    If $0 < \mu(A_n \cap \supp(\rho(w))) < \mu(A_n)$ there is nothing to do; otherwise there is a case distinction.

    If $\mu(A_n \cap \supp(\rho(w))) = 0$, choose $w \in \Gamma$ and $\tau \in \Hom_\sigma(\Sym(\Z))$ such that $\tau(w)(0) \neq 0$.
    If $\mu(A_n \cap \supp(\rho(w))) = \mu(A_n)$, choose $w \in \Gamma$ and $\tau \in \Hom_\sigma(\Sym(\Z))$ such that $\tau(w)(0) = 0$.
    Apply Lemma \ref{lem:instert words into constraineds} with $\pi_i = \tau(b_i)$ and $A = A_n$.
    Since $\epsilon < \mu(A_n)$, in either case, the representation $\rho'$ obtained satisfies $0 < \mu(A_n \cap \supp(\rho'(w))) < \mu(A_n)$.
\end{proof}
\end{no}

\section{Generic restricted representation are dense}\label{SecDense}

\subsection{A theorem of Krengel, revisited}

\begin{no}\label{ThmKrengel} Let $(X, \mu)$ be a standard probability space and $T \in \mathrm{Aut}(X, [\mu])$. We say that $A \in \malg(X, \mu)$ is \emph{transitive} if its $T$-orbit $\cO_A \coloneqq  \{T^n(A) \mid n \in \Z\}$ is dense in $ \malg(X, \mu)$.
We denote by $\malg(X, \mu)^{\mathrm{trans}} \subset \malg(X, \mu)$ the subset of transitive points. 

Note that if $[\mu]$ contains a $T$-invariant probability measure $\nu$, then $\nu(B) =\nu(A)$ for all $A \in \malg(X, \mu)$ and all $B$ in the orbit closure of $A$. Thus if $\malg(X, \mu)^{\mathrm{trans}}\neq \emptyset$, then $[\mu]$ cannot contain an invariant probability measure, and hence $T$ must be of Krieger type  $\I_\infty, \II_\infty$ or $\III_\lambda$ by Proposition \ref{WWTypes1}. The following result of Krengel provides a very strong converse (under the assumption that $T$ is ergodic); a more general version of this theorem apparently first appears explicitly as \cite[Theorem 3.1]{JonesKrengel1974} in a paper of Jones and Krengel, but as Krengel explains on \cite[page 294]{JonesKrengel1974}, the theorem follows from his earlier paper \cite{Krengel1970}. In any case, it seems that the theorem is not as well-known as it should be.
\begin{theorem*}[Krengel] \label{thm:jknew}
   If $T \in \ns$ is ergodic of type $\I_\infty$, $\II_\infty$ or $\III_\lambda$, then $\malg(X, \mu)^{\mathrm{trans}} \subset \malg(X, \mu)$ is a dense $G_\delta$-subset.
\end{theorem*}
\end{no}
\begin{no} The fact that $\malg(X, \mu)^{\mathrm{trans}} \subset \malg(X, \mu)$ is a $G_\delta$-set is actually not stated in \cite{JonesKrengel1974}, but it is obvious: Given $U \in \malg(X, \mu)$ and $\delta>0$ we define an open subset $\Omega_{U,\delta} \coloneqq  \{A \in \malg(X,\mu) \mid \exists\, n \in \Z \colon \mu(T^n A \triangle U ) < \delta \} 
\subset  \malg(X,\mu)$ and observe that
  \[
   \malg(X, \mu)^{\mathrm{trans}} =  \bigcap_{U \in \cU} \bigcap_{n \in \bN} \Omega_{U,\frac{1}{n}},
 \]  
 where $\cU$ is some countable basis for the $\sigma$-algebra of $X$. Using this representation of $ \malg(X, \mu)^{\mathrm{trans}}$ we can give a quick proof of Theorem \ref{thm:jknew} via Baire category.
 For this let $T$ as in Theorem \ref{thm:jknew} and define
 \[
\cV(T) \coloneqq  \left\{A \in \malg(X,[\mu]) \, \Big| \, \lim_{n \to \infty} \frac{1}{n} \sum_{k=0}^{n-1} \mu(T^k A) = 0 \right\}.
\]
\begin{lemma*}\label{lem:Vplus1}
    The set $\cV(T)$ is closed under finite unions and passing to subsets and contains all weakly wandering sets for $T$.
\end{lemma*}

\begin{proof}{
That $\cV(T)$ is closed under finite unions and passing to subsets is immediate from additivity of $\mu$. Now let $A$ be weakly wandering for $T$ and abbreviate $A_j \coloneqq T^j A$ for all $j \in \bZ$. Fix $r>0$ and choose $k_1,\dots,k_{r-1}$ such that $A_0,A_{k_1},\dots,A_{k_{r-1}}$ are disjoint up to a nullset. For some large $n$, consider $r$ disjoint copies of the first $n$ translates of the set $A$ and arrange these in an $r \times n$ diagram, as depicted in Figure \ref{fig:weak_wand} for the case where $r=4$, and $k_1=1,k_2=3,k_3=5$. 
\begin{figure}[h!]
            \centering
            \input{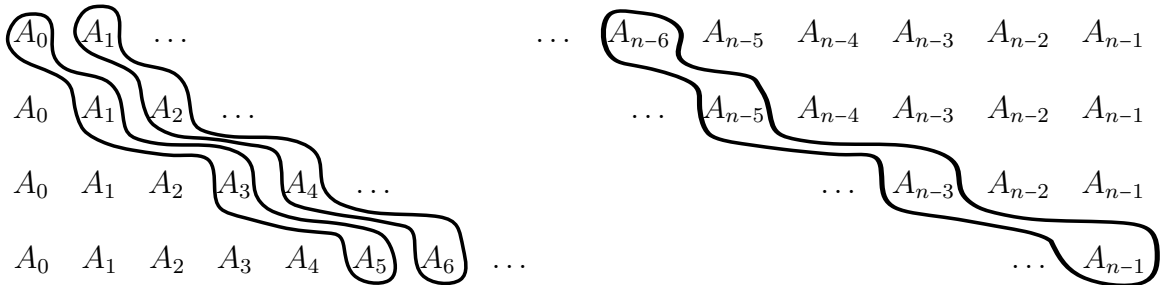} 
    \caption{Weakly wandering at steps $0,1,3$ and $5$.}
    \label{fig:weak_wand}
\end{figure}

Since the encircled collections denote disjoint sets we bound the total measure of the sets appearing in each of the $n-k_{r-1}$ encircled areas by $1$. Each of the $r k_{r-1}$ remaining sets also has measure at most $1$, yielding the following estimate: 
\begin{align*}
r \cdot \sum_{k=0}^{n-1} \mu(A_k)  \leq (n - k_{r-1}) + rk_{r-1}.
\end{align*}
Dividing by  $nr$ and taking the $\limsup$ over $n$ yields
\[
\limsup_{n \to \infty} \frac{1}{n} \sum_{k=0}^{n-1} \mu(T^k A) \leq \frac{1}{r} \quad \text{for every }r \in \bN.
\]
Finally, letting $r \to \infty$ proves that $A \in \cV(T)$.}
\end{proof}

\begin{corollary*}\label{cor:Vplusdense}
For all $A \in \cB$ and $\varepsilon > 0$ there exists $V \in \cV(T)$ with $V \subset A$ and $\mu(A \setminus V) < \varepsilon$. In particular, $\cV(T)$ is dense in $\malg(X, [\mu])$.
\end{corollary*}
\begin{proof}  By Proposition \ref{WWTypes1}, $T$ admits a weakly wandering set $W$ of positive measure, and since $T$ acts ergodically, we have $\smash{\mu\left(\bigcup_{n \in \bZ} T^n W\right)=1}$. Hence there exists a finite set $I \subset \bZ$ with $\mu(A \setminus \bigcup_{i \in I} T^i W ) < \varepsilon$. Now $W$ and its translates are weakly wandering for $T$, hence by Lemma \ref{lem:Vplus1} the subset $V \coloneqq  A \cap \bigcup_{i \in I} T^i W \subset A$ is in $\cV(T)$.
\end{proof}
\end{no}

\begin{no}
Recall that a subset $M \subset \bN$ has \emph{natural density} $\kappa \in \bR$ if $\lim_{n \to \infty} \frac{|M \cap [0,n-1]|}{n} = \kappa$.
\begin{lemma*}\label{LemmaDensity1} If $A \in  \cV(T)$, then the set
$M(T, A, \delta) \coloneqq  \{n \in \bN_0 \mid \mu(T^n A) < \delta\}$ has natural density $1$ for every $\delta>0$.
\end{lemma*}
\begin{proof} If $M^c$ denotes the complement of $ M(T, A, \delta)$, then for every $n \in \bN$ we have
\begin{align*}
 \sum_{k = 0}^{n-1} \mu(T^{k} A) \geq \sum_{k \in M^c \cap [0, n-1]} \mu(T^k A) \geq \sum_{k \in M^c \cap [0, n-1]} \delta \geq |M^c \cap [0,n-1]| \cdot \delta.
\end{align*}
Divide by $n\delta$ and let $n \to \infty$; as $A \in \cV(T)$ we deduce that $M^c$ has natural density $0$.
\end{proof}
\end{no}

\begin{no}
\begin{proof}[Proof of Theorem \ref{thm:jknew}] Fix $U \in \malg(X, \mu)$ and $\delta > 0$; by Baire category we need to show that $\Omega_{U,\delta} \subset \malg(X, \mu)$ is dense. Fix $B \in \cB$ and $\varepsilon \in (0, \delta)$.

\item Applying Corollary \ref{cor:Vplusdense} to both $T$ and $T^{-1}$ we find $U' \subset U$ and $B' \subset B$ such that
\[
U' \in \cV(T), \quad \mu(U \setminus U') < {\varepsilon}/{2}, \quad B' \in \cV(T^{-1}) \qand \mu(B \setminus B') < {\varepsilon}/{2}.
\] 
By Lemma \ref{LemmaDensity1} the sets  $M(T, B', \varepsilon/2)$ and $M(T^{-1}, U', \varepsilon/2)$ have natural density $1$, hence intersect in a set $M$ of natural density $1$. For all $n \in M$ we have
$\mu(T^{-n}U') < {\varepsilon}/{2}$ and $\mu(T^n B') < {\varepsilon}/{2}$, hence if we set $A \coloneqq  B' \cup T^{-n}U'$, then
\[
\mu(T^n A \triangle U) \leq \mu(U \setminus U') + \mu(T^nB') < \varepsilon < \delta \implies A \in \Omega_{U,\delta},
\]
and $\mu(A \triangle B) \leq \mu(B \setminus B') + \mu(T^{-n}U') < \varepsilon$. This shows that $ \Omega_{U,\delta}$ is dense. 
\end{proof}
\end{no}
\subsection{Density in the measurable full group}\label{SecDensity}
\begin{no}
Although we formulate the following proposition with the assumption that $[\cR]$ is hyperfinite, we thank Fran\c{c}ois Le Ma\^itre for pointing out that this is not necessary and we actually show that the measurable full group of any non-p.m.p. ergodic countable Borel equivalence relation is topologically 2-generated.

\begin{proposition*}\label{prop:dense}
    Assume $\cR$ is of type $\II_\infty$ or $\III_\lambda$.
    Then the subset $\mathrm{Hom}^{\mathrm{den}}_T(\Gamma, G) \subset \mathrm{Hom}_T(\Gamma, G)$ of constrained representations with dense image is a dense $G_\delta$ set.
\end{proposition*}
\begin{proof} Fix a countable dense subset $(g_i)_{i \in I}$ in $G$ and observe that
\[
\mathrm{Hom}^{\mathrm{den}}_T(\Gamma, G) = \bigcap_n \bigcap_i \Omega_{n,i}, \; \text{where} \; \Omega_{n,i} \coloneqq  \big\{\rho \in \Hom_T(\Gamma,G) \, \big| \, \exists w \in \Gamma \colon d(\rho(w),g_i) < \frac{1}{n} \big\}.
\]
Now fix $n \in \bN$ and $i \in I$. Given $w \in \Gamma$ we define \[f_{w,i} \colon \Hom_T(\Gamma,G) \to \mathbb{R}_{\geq 0}, \quad  f_w(\rho) \coloneqq  d(\rho(w),g_i).\] By continuity of $d$ and the evaluation maps these maps are continuous, and thus
\[
\Omega_{n,i} = \bigcup_{w \in \Gamma}  f_{w,i}^{-1}\left(\big[0,\frac{1}{n}\big)\right)
\]
is open; it remains to prove that it is also dense.

For this we fix $\rho \in \Hom_T(\Gamma,G)$ and $\eps \in (0, \min\{1/10, 1/n\})$. By Theorem \ref{thm:jknew} we may choose $A \in \malg(X, \mu)^{\mathrm{trans}}$; since $\mu$ is non-atomic we can assume without loss of generality that $\epsilon < \mu(X \setminus A)<2\epsilon$. Since $A$ is transitive, we can now choose $n_1,n_2,n_3 \in \Z$ satisfying the following conditions:
\begin{enumerate}[(i)]
    \item $\mu(T^{n_1}(A)) < \epsilon$.
    \item \label{itm:close to g} $\mu(A \triangle g_i^{-1}(T^{n_2}(A))) < \epsilon$.
    \item \label{itm:small image} $\mu(T^{n_2+n_3}(A))<1-\mu(\rho(b_1)(A))$.
    \item $\mu(\rho(b_1)^{-1}(T^{n_2+n_3}(A)))<\epsilon$.
\end{enumerate}
Now define a partial Borel automorphism $\tau\colon \dom(\tau) \to \ran(\tau)$ of $X$ by
\begin{equation}\label{equ:tau}
\tau(x) = \begin{cases}
                \rho(b_1)(x), & \text{if } x \in D_1 \coloneqq  A \setminus (T^{n_1}(A) \cup \rho(b_1)^{-1}(T^{n_2+n_3}(A)))\\
                T^{n_3}(g_i(T^{-n_1}(x))), & \text{if }x \in D_2 \coloneqq  T^{n_1}(A \cap g_i^{-1}(T^{n_2}(A))).
            \end{cases}
\end{equation}

This is well-defined, since $D_1 \cap D_2 = \emptyset$, and it is injective since $T^{n_3}(g_i(T^{-n_1}(T^{n_1}(D_2))))\subset T^{n_3+n_2}(A)$ and the latter is disjoint from $\rho(b_1)(D_1)$.

\begin{figure}
    \centering
    \includegraphics[scale=1]{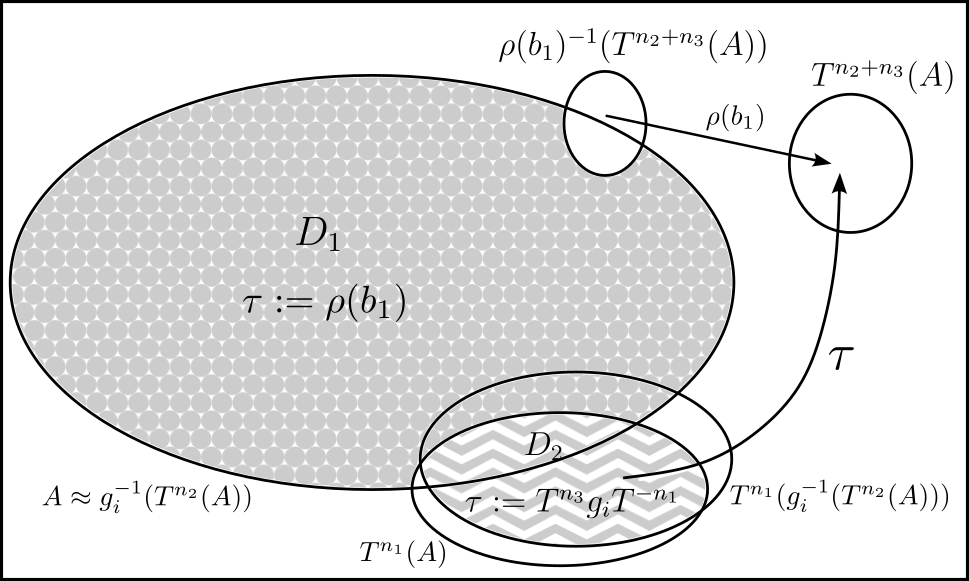}
    \caption{Depiction of $\tau$ as in (\ref{equ:tau}).}
    \label{fig:tau}
\end{figure}

Next we observe that by (i) and (iv) and the choice of $A$ we have
\begin{align*}
0 & < (1-2\epsilon)-2\epsilon < \mu(A)-\mu(T^{n_1}(A)) - \mu(\rho(b_1)^{-1}(T^{n_2+n_3}(A))) \leq \mu(D_1)\\
 & \leq \mu(\dom(\tau)) \leq  \mu(A) +\mu(T^{n_1}(A)) < (1 -\epsilon) + \epsilon = 1.
\end{align*}
Since $\rho(b_1)$ is measure-class preserving and $\mu(D_1) > 0$ we deduce with (iii) that
\begin{align*}
0 & < \mu(\rho(b)(D_1)) \leq \mu(\ran(\tau)) \leq \mu( \rho(b_1)(A) \cup T^{n_3}(g_i(A)\cap T^{n_2}(A)))\\ & \leq  \mu(\rho(b_1)(A)) + \mu(T^{n_2+n_3}(A))<1.
\end{align*}
We can thus apply Lemma \ref{lem:extension_infty} to find $\tau' \in G$ extending $\tau$. 

\item We now define $\eta \in \Hom_T(\Gamma,G)$ by setting $\eta(b_1) = \tau'$ and $\eta(b_i)=\rho(b_i)$ for all $i \geq 2$.
If we set  $w \coloneqq  a^{-n_3} b_1 a^{n_1}$, then for all  $x \in A \cap g_i^{-1}(T^{n_2}(A))$ we have $\eta(w)(x)=g_i(x)$; it thus follows from (ii) that $d(g_i,\eta(w))<\epsilon \leq 1/n$, and hence $\eta \in \Omega_{i,n}$ since 
$d(\eta(b_1),\rho(b_1))\leq 4\epsilon$. This proves that $\Omega_{i,n}$ is dense in $\Hom_T(\Gamma,G)$.
\end{proof}
\begin{remark*}\label{DenseGeneral} 
    Krengel's theorem obviously cannot hold in type $\II_1$.
    It is an open question whether constrained representations into full groups of Krieger type $\II_1$ are dense for all possible choices of generators $T$.
    In his PhD thesis, Le Ma\^{i}tre shows that it is true if $T$ is a rank one transformation (\cite[Theorem 5.4]{LeMaitrePhD}), but for general $T$ it is unknown, see \cite[Question 1.9]{LM:full_groups} and the remark following it.
\end{remark*}
\end{no}

\subsection{Consequences of density}

\begin{no}\label{par:high trans from dense}

Recall that for subgroups of $\Sym(\Z)$ high transitivity is equivalent to density.
One implication also holds in the non-atomic case.

\begin{proposition*}
    Let $\Lambda \leq G$ be dense. Then $\Lambda$ acts highly transitively on almost every orbit.
\end{proposition*}

\begin{proof}
    Let $n \in \bN$ and $\pi \in \Sym_{n+1}$.
    For $g \in G$ we say that $g_{x,n} = \pi$ 
    if $\alpha(g,x)|_{[0,n]} = \pi$.
    Let $\epsilon > 0$.
   It suffices to show that
    \[
    \mu(\bigcup_{\gamma \in \Lambda} \{x \in X \mid \gamma_{x,n} = \pi\}) > 1-2\epsilon.
    \]
    There exists $g \in G$ with $\mu(\{x \in X \mid g_{x,n} = \pi\}) > 0$.
    By density, there exists $\gamma \in \Lambda$ and $A \subset X$ with $\mu(A) > 0$ and
    $\gamma_{x,n} = \pi$ for all $x \in A$.
    By ergodicity of $T$ we know that $X = \bigcup_{k \in \bN} T^k A$.
    The idea is now the following:
    If a transformation $T_k \in \Lambda$ is very close to $T^k$, then on a large part of $T^kA$ we will have $(T_k \gamma T_k^{-1})_{x,n} = \pi$.
    
    Recall that $d(g,h) = \mu(\{x \in X \mid g(x) \neq h(x)\})$.
    Let $k \in \bN$ and choose $T_k \in \Lambda$ with $d(T_k^{-1} T^{n'},T^{-k+n'}) < \frac{\epsilon}{2^k (n+1)}$ for all $0 \leq n' \leq n$. This implies that
    $\mu(\{x \in X \mid \exists 0 \leq n' \leq n \colon T_k^{-1}(T^{n'}x) \neq T^{-k+n'}x \}) < \frac{\epsilon}{2^k}$.

    Set $\gamma_k := T_k \gamma T_k^{-1}$
    Now
    \begin{align*}
        \{x \in T^k A \mid (\gamma_k)_{x,n} = \pi \} &\supset \{x \in T^k A \mid \forall 0 \leq n' \leq n \colon T_k^{-1}(T^{n'}x) = T^{-k}(T^{n'}x) \} =: A_k
    \end{align*}
    and by the above $\mu(T^kA \setminus A_k) \leq \frac{\epsilon}{2^k}$.
    Summarizing
    \begin{align*}
     \mu(\bigcup_{\gamma \in \Lambda} \{x \in X \mid \gamma_{x,n} = \pi\}) &\geq \mu(\bigcup_{k \geq 0} \{x \in X \mid (\gamma_k)_{x,n} = \pi\})    \\
    &\geq \mu(\bigcup_{k \geq 0} A_k ) \\
    &\geq 1 - \sum_{k \geq 0} \mu(T^kA \setminus A_k) \\
     & > 1 - \sum_{k \geq 0} \frac{\epsilon}{2^k} = 1 - 2\epsilon.
    \end{align*}
\end{proof}

In \cite[Sect. 1.2]{LM:full_groups} Le Ma\^itre gives an example showing that highly transitive subgroups of measurable full goups are not necessarily dense, but the equivalence relation there is not ergodic. He explained to us also an example with an ergodic equivalence relation.

\begin{remark*} Proposition \ref{par:high trans from dense} implies the missing part of Theorem \ref{GCR1}\eqref{itm:nonatomic afht} if $\cR$ is of type $\II_\infty$ or $\III_{\lambda}$. This argument does not apply in the $\II_1$-case, since we do not know whether generic constrained representations are dense in this case. Nevertheless, Proposition \ref{par:high trans from dense} (and hence Theorem \ref{GCR1}\eqref{itm:nonatomic afht}) also holds in the $\II_1$-case by a different argument, see \cite[Proposition 1.19]{EG:generic_irs}. The latter actually applies  for all p.m.p. ergodic equivalence relations, even beyond the hyperfinite case.
\end{remark*}
    
\end{no} The following result was established in \cite[Prop.\ 2.4]{LM18} in case $\II_1$, but the same proof works in Krieger types $\II_\infty$ and $\III_\lambda$.
\begin{proposition*}\label{tnf}
  Let $\Lambda < G$ be a countable, dense group. Then, the action of $\Lambda$ on $X$ is totally-non-free.
\end{proposition*}
In view of \S \ref{prop:dense} this yields an alternative proof of Proposition \ref{prop-tnf}.
\end{no}



    
    

\section{Proofs of the theorems from the introduction}\label{SecCollect}

We now explain how the results above imply the theorems from the introduction. 

\begin{proof}[Proof of Theorem \ref{GCR1}]
We have established \eqref{itm:nonatomic ec} in \S \ref{ecproof} and \eqref{itm:nonatomic inj}, and \eqref{itm:nonatomic afht} and \eqref{itm:nonatomic tnf} in Subsection \ref{SecFurther}. If $\cR$ has Krieger type different from $\II_1$, then there is no $T$-invariant probability measure in $[\mu]$, and hence it follows from \eqref{itm:nonatomic tnf} that there is no invariant probability measure in the measure class of $\mathrm{Stab}_*\mu$; this proves \eqref{NoInvMeasure}. Finally, \eqref{itm:nonatomic dense} was established in Proposition \ref{prop:dense}. 
\end{proof}
\begin{proof}[Proof of Theorem \ref{ExECRS}] Given $\lambda \in [0,1]$ we pick a system $(X_\lambda, \mu_\lambda, T_\lambda)$ of Krieger type $\III_\lambda$.
We then denote by $\mu^\lambda \in \mathrm{Prob}(\Sub(\Gamma))$ the push-forward of $\mu_\lambda $ under the corresponding stabilizer map, which is $\Gamma$-equivariant.
Note that pushforwards of ergodic measures under equivariant maps remain ergodic; the same holds for nonsingular and conservative actions.
By Theorem \ref{GCR1}(\ref{itm:nonatomic inj}), 
the conjugation action of $a$ on $(\Sub(\Gamma), \mu^{\lambda})$ is of Krieger type $\III_\lambda$. This implies that the measures $\mu^\lambda$ are mutually singular and singular to every $a$-invariant probability measure on $\Sub(\Gamma)$, and in particular to every IRS. By Theorem \ref{GCR1}(\ref{itm:nonatomic ec}) the action of $\Gamma$ on $(\Sub(\Gamma), \mu^{\lambda})$ is then elementwise conservative, and by Theorem \ref{GCR1}(\ref{itm:nonatomic afht}) the action is supported on core-free, co-highly transitive and co-amenable subgroups.
\end{proof}

\begin{proof}[Proof of Theorem \ref{thm:S_infty}]
This has been established in Section \ref{SecIInf}.
\end{proof}

\begin{proof}[Proof of Theorem \ref{thm:unrestricted}]
This theorem is actually easier than Theorem \ref{GCR1} and could be proved along similar lines. However, since we already have established Theorem \ref{GCR1}, we can also use it to quickly deduce Theorem \ref{thm:unrestricted} as follows:

Apply Theorem \ref{GCR1} to the group $\tilde{\Gamma} = \Z * \Gamma$. Via the natural isomorphism $\Hom_T(\tilde{\Gamma},G) \cong \Hom(\Gamma,G)$ almost everything we want follows upon taking a generic $\rho \in \Hom_T(\tilde{\Gamma},G)$ and then ignoring the first generator by restricting $\rho$ to $\Gamma$. 

The only things that do not follow directly are the density statement (\ref{itm:p_dense})
and the high transitivity clause in (\ref{itm:p_prop}). By Proposition \ref{par:high trans from dense} it is enough to establish the former. But the density has nothing to do with group theory. A generic countable subset of any Polish space, and in particular the $\rho$-image of the set of the free generators $H(\rho) = \{\rho(b_1),\rho(b_2),\ldots\} \subset G$, is already dense. Indeed for a
given open $U \subset G$ the set $\Omega(U) = \{\rho \ | \ H(\rho) \cap U\}$ is dense and open in $\Hom(\Gamma,G)$. The desired density follows from Baire's category theorem by intersecting $\bigcap_{U \in \cU} \Omega(U)$, over a countable basis $\cU$ for the topology of $G$.
\end{proof}

 \bibliographystyle{alpha}
 \bibliography{rec}
\end{document}